\documentclass[10pt, reqno]{amsart}
\usepackage[utf8]{inputenc}
\usepackage[english]{babel}

\usepackage{amsmath,amssymb,amsfonts}
\usepackage{amsthm}
\usepackage{float}
\usepackage{url}
\usepackage[hidelinks]{hyperref}
    
\theoremstyle{plain}
\newtheorem{theorem}{Theorem}[section]
\newtheorem{prop}[theorem]{Proposition}
\newtheorem{lemma}[theorem]{Lemma}
\newtheorem{corollary}[theorem]{Corollary}
\theoremstyle{definition}
\newtheorem{definition}[theorem]{Definition}
\newtheorem{example}[theorem]{Example}
\newtheorem{remark}[theorem]{Remark}

\usepackage{mathtools}
\usepackage{enumitem}
\usepackage[mathscr]{euscript}

\title[Spectrality of induced operations]{Closure operators on semilattice-ordered semigroups and spectrality of induced operations}

\author[D. Siejwa]{Damian Siejwa}
\address{Faculty of Mathematics and Information Science\\
Warsaw University of Technology\\
00-661 Warsaw, Poland}

\email{damian.siejwa.dokt@pw.edu.pl}


\begin{document}
\begin{abstract}
    Let $S$ be a semilattice-ordered semigroup and let $\mathrm{cl}$ be a~closure operator on $S$. We consider the space $$X := \{A \in \mathcal{P}(S) \mid A^\mathrm{cl} = A\}$$ of all $\mathrm{cl}$-closed subsets of $S$, endowed with the subspace topology induced by the hull-kernel topology on $\mathcal{P}(S)$. We prove that $X$ is a spectral space and a retrocompact subset of $\mathcal{P}(S)$ if and only if $\mathrm{cl}$ is algebraic. Assuming that $\mathrm{cl}$ is algebraic, we then investigate the operation $$A \star B := (AB)^\mathrm{cl}$$ induced on $X$ by the multiplication on $S$. We obtain finite characterizations of spectrality of the map $\star: X \times X \to X$ in terms of finite subsets of $S$ and finitely generated $\mathrm{cl}$-closed subsets. Analogous criteria are established for left and right translations. We prove that if the underlying poset $(S, \leqslant)$ is well-quasi-ordered and the closure operator $\mathrm{cl}$ is algebraic and order-compatible, then the induced operation $\star$ is spectral. For the identity closure operator, we characterize spectrality of $\star$ by the finite decomposition property. We also show that spectrality of $\star$ implies spectrality of all left and right translations, whereas the converse fails in general, even for an algebraic multiplicative closure operator. Finally, we introduce three closure operators naturally associated with semilattice-ordered semigroups, study their algebraic and multiplicative properties, and apply the general results to the corresponding spectral spaces and induced operations.  
\end{abstract}
\maketitle
\section{Introduction}
Many constructions in algebra lead naturally to topological spaces whose points represent algebraic objects ordered by inclusion. The prime spectrum of a commutative ring is the guiding example: closed sets are described by containment, the space is quasi-compact, and the specialization relation reflects inclusion among prime ideals. Hochster provided a topological characterization of such spectra by introducing the notion of a spectral space, which is a quasi-compact, $T_0$, sober topological space such that the family of all quasi-compact open subsets is closed under finite intersections and forms a basis of open subsets \cite{Hochster}. Spectral spaces have since become a standard object of study in both topology and algebra. 

Although the prime spectrum endowed with the hull-kernel topology (also called the Zariski topology) is always quasi-compact and $T_0$, it is in general far from being Hausdorff (it is Hausdorff only in the zero-dimensional case \cite[Theorem 1.3]{Maroscia}). Many authors therefore consider a finer topology on the same underlying set, namely the patch (or constructible) topology. It is the coarsest refinement of the spectral topology in which all quasi-compact open subsets become clopen. One reason for its importance is that it provides a canonical quasi-compact Hausdorff topology. A further reason is that the patch topology interacts tightly with the specialization order: a spectral space is uniquely determined by its patch topology together with its specialization preorder. This fact is the basis for an alternative approach to spectral spaces via Priestley spaces and Priestley duality \cite{Dickmann_Schwartz_Tressl_2019, Priestley, Craig, JungRivieccio}. 

Closure operators provide a natural framework for constructing and studying spectral spaces. Given a closure operator on a family of subsets, one may consider its fixed points as a space ordered by inclusion and equipped with the hull-kernel topology. This perspective appears in several algebraic settings \cite{JunRayTolliver, FinocchiaroFontanaSpiritoNullstellensatz, FinocchiaroFontanaSpirito, Banerjee}, where some additional finite type assumptions ensure that the resulting space is spectral.

The present paper develops this point of view for semilattice-ordered semigroups. These structures are closely related to additively idempotent semirings, where the idempotent addition is typically interpreted as a join. However, we do not require an additive identity to exist. A related closure-theoretic approach in that setting was developed by Jun, Ray and Tolliver \cite{JunRayTolliver}. Topological properties of semilattice-ordered semigroups were already studied by Calude \cite{calude1976_1}. More recently, hull-kernel topologies on prime ideals in ordered semigroups have been investigated in \cite{huanrong, luangchaisri_changphas}, continuing the general idea that spectrum spaces with hull-kernel topology remain meaningful beyond commutative rings. 

Let $S$ be a semilattice-ordered semigroup and let $\mathrm{cl}$ be a closure operator on $S$. We consider the set $$X := \{A \in \mathcal{P}(S) \mid A^\mathrm{cl} = A\},$$ endowed with the subspace topology induced by the hull-kernel topology on $\mathcal{P}(S)$. We first establish the topological properties of this space. Namely, we show that $X$ is a spectral space and a retrocompact subset of $\mathcal{P}(S)$ if and only if the closure operator $\mathrm{cl}$ is algebraic (see Theorem \ref{construction_of_spectral_space}). This result provides the topological setting in which we subsequently study operations induced on $X$ by the multiplication of $S$.

Assuming that $\mathrm{cl}$ is algebraic, we consider the binary operation $$\star: X \times X \to X, \qquad A\star B := (AB)^\mathrm{cl}.$$ We obtain several equivalent finite characterizations of spectrality of $\star$ (see Corollary \ref{equivalent_conditions_on_spectrality}) and derive sufficient conditions based on well-quasi-ordering (see Proposition \ref{well_quasi_order_implies_spectrality} and Corollary \ref{algebraicity_order_compatibility_and_well_quasi_order_on_original_set_imply_spectrality_of_star}). We also investigate the relation between spectrality of $\star$ and spectrality of its left and right translations. In the case of the identity closure operator, we characterize spectrality of $\star$ in terms of the finite decomposition property of $S$ (see Theorem \ref{spectrality_is_equivalent_to_finite_decomposition_property}). 

\medskip 
\noindent\textbf{Plan of the paper.} Section \ref{background} contains the necessary preliminaries on semilattice-ordered semigroups, closure operators and spectral spaces. In particular, we recall the hull-kernel topology on $\mathcal{P}(S)$, the patch topology on a spectral space and the ultrafilter criterion for spectrality. 

Section \ref{spectral_spaces_arising_from_closures} is devoted to the space $X$. We show that algebraicity of the closure operator $\mathrm{cl}$ is equivalent to spectrality and retrocompactness of the associated space of $\mathrm{cl}$-closed sets. We also give an example showing that the algebraicity assumption is essential. 

In Section \ref{topological_semigroups_whose_underlying_space_is_spectral} we assume that $\mathrm{cl}$ is algebraic and study spectrality of the operation $\star~: X~\times~X~\rightarrow X$ induced by the multiplication on $S$. We derive finite characterizations of spectrality of $\star$, establish analogous criteria for left and right translations, and obtain sufficient conditions for spectrality from well-quasi-order assumptions. We also consider the case in which $(X, \star)$ is a semilattice. Finally, for the identity closure operator, we characterize spectrality of $\star$ by the finite decomposition property and obtain corresponding criteria for spectrality of translations. 

In Section \ref{closures_on_semigroups} we introduce three closure operators naturally associated with semilattice-ordered semigroups: the downward-join closure, the join-radical closure and the join-bi-ideal closure. We establish their basic properties and determine which of them are algebraic and which are multiplicative.

Finally, in Section \ref{applications} we apply the general results proved earlier to the closure operators introduced in Section \ref{closures_on_semigroups}. We thereby obtain concrete examples of spectral spaces of $\mathrm{cl}$-closed subsets. In particular, we show that spectrality of all left and right translations does not imply spectrality of $\star$ in general, even for an algebraic and multiplicative closure operator. 

\medskip
\noindent\textbf{Notational conventions.} Throughout the paper, $\mathbb{Z}$ denotes the set of all integers, $\mathbb{N}$ the set of all nonnegative integers and $\mathbb{N}_+$ the set of all positive integers. For a set $S$ we write $\mathcal{P}(S)$ for its power set and $\mathcal{P}_\mathrm{fin}(S)$ for the family of all finite subsets of $S$. If $f : X \rightarrow Y$ is a map, $A \subseteq X$ and $B \subseteq Y$, then $f(A)$ denotes the image of $A$ under $f$, while $f^{-1}(B)$ denotes the preimage of $B$ under $f$. 
\section{Background}\label{background}  
In this section we recall basic notions and facts that will be used throughout the paper. We begin with semilattice-ordered semigroups and the notation associated with the partial order arising from the semilattice structure. We then provide the main definitions and facts about spectral spaces, patch and hull-kernel topologies. Readers familiar with these topics may skip to the next section. 
\subsection{Semilattice-ordered semigroups}
A \emph{semilattice-ordered semigroup} is a triple $(S, \cdot, \lor)$, where $S$ is a nonempty set, $(S, \cdot)$ is a semigroup, $(S, \lor)$ is a semilattice and $$x(y \lor z) = xy \lor xz, \qquad (x \lor y)z = xz \lor yz$$ for all $x,y,z \in S$. Any semilattice-ordered semigroup is a partially ordered set in which the binary relation $x \leqslant y$ is defined as follows: $$x \leqslant y \iff x \lor y = y.$$ It is easy to see that if $x \leqslant y$ and $u \leqslant v$, then $x \lor u \leqslant y \lor v$ and $xu \leqslant yv$. For subsets $A, B\subseteq S$ we write $AB := \{ab \mid a\in A, b\in B\}$. 

\medskip \noindent 
\textbf{Relation to additively idempotent semirings.} Semilattice-ordered semigroups are closely related to \emph{additively idempotent semirings}, but they are more general. Indeed, if $(R, +, \cdot)$ is an additively idempotent semiring, then the addition defines a semilattice operation by $x \lor y := x + y$, and $(R, \cdot, \lor)$ is a semilattice-ordered semigroup. However, the converse does not hold in general, since a semiring also requires a distinguished additive identity. 
\subsection{Closure operators}
Let $S$ be a semilattice-ordered semigroup. A \emph{closure operator} on $S$ is a map $\mathrm{cl} : \mathcal{P}(S) \rightarrow \mathcal{P}(S)$ such that for all $A, B \in\mathcal{P}(S)$: 
\begin{enumerate}[label=(\arabic*)]
    \item (\emph{Extension}) $A \subseteq\mathrm{cl}(A)$, 
    \item (\emph{Idempotence}) $\mathrm{cl}(A) = \mathrm{cl}(\mathrm{cl}(A))$, 
    \item (\emph{Order-preservation}) $A \subseteq B \implies \mathrm{cl}(A) \subseteq \mathrm{cl}(B)$. 
\end{enumerate}
For notational convenience, we write $A^{\mathrm{cl}} := \mathrm{cl}(A)$ for every $A \in \mathcal{P}(S)$. A subset $A \in \mathcal{P}(S)$ is said to be \emph{$\mathrm{cl}$-closed} if $A^\mathrm{cl} = A$. A closure operator $\mathrm{cl}$ on $S$ is said to be \emph{algebraic} (or \emph{of finite type}) if $A^{\mathrm{cl}} = \bigcup\{F^{\mathrm{cl}} \mid F\in\mathcal{P}_{\mathrm{fin}}(A)\}$ for any $A \in \mathcal{P}(S)$. 
\subsection{Spectral spaces}
We now recall several standard facts concerning spectral spaces. Our aim is not to give a comprehensive account, but rather to collect the specific tools needed later. 
\begin{definition}\cite[Proposition 4]{Hochster}
    A \emph{spectral space} is a quasi-compact and $T_0$ topological space~$X$ such that: 
    \begin{enumerate}[label=(\arabic*)]
        \item any finite intersection of quasi-compact open subsets of $X$ is a quasi-compact open subset and the family of all quasi-compact open subsets of $X$ forms a basis of $X$, 
        \item any nonempty irreducible closed subset of $X$ has a unique generic point. 
    \end{enumerate}
    A topology on a spectral space $X$ is called a \emph{spectral topology}. For more information about spectral spaces, see \cite{Dickmann_Schwartz_Tressl_2019}.
\end{definition}
\begin{definition}\cite[Section 2]{Finocchiaro2019SupremaIS}
    Let $X, Y$ be spectral spaces. 
    \begin{enumerate}[label=(\arabic*)]
        \item A map $f : X \rightarrow Y$ is said to be \emph{spectral} if for every quasi-compact open subset $V$ of $Y$, $f^{-1}(V)$ is quasi-compact and open.
        \item A subset $W$ of $X$ is said to be \emph{retrocompact} in $X$ if for every quasi-compact open subset $U$ of $X$, $U \cap W$ is quasi-compact.
        \item A \emph{patch} (or \emph{constructible}) \emph{topology} on $X$ is the coarsest topology for which all quasi-compact open subsets of $X$ are clopen. We denote by $X^\mathrm{cons}$ the space $X$ with the patch topology.
    \end{enumerate}  
\end{definition} 
\noindent The patch topology plays a central role in the modern theory of spectral spaces: it is the coarsest refinement in which quasi-compact opens become clopen, and it is particularly well-suited for continuity questions about algebraically defined operations. We shall use it in a similar spirit as in \cite{Finocchiaro2019SupremaIS}, where patch topologies and spectral maps are used to relate an algebraic structure to topological properties. 

Now we provide some properties of the patch topology.
\begin{lemma}\cite[Section 08YF]{stacks-project}\label{spectrality_implies_continuity_in_the_patch_topology}
    Let $f: X \to Y$ be a spectral map of spectral spaces. Then $f$ is continuous in the patch topology. 
\end{lemma}
\begin{theorem}\cite[Theorem 1]{Hochster}\label{spectral_space_with_patch_top_is_compact}
    Let $X$ be a spectral space. Then the space $X$ with the patch topology is quasi-compact and Hausdorff. 
\end{theorem}
For a spectral space $X$, the product $X \times X$ is again a spectral space. Moreover, the patch topology on $X \times X$ coincides with the product topology of the patch topologies on $X$ \cite[Section 2]{PICAVET1986527}, i.e. $$(X \times X)^\mathrm{cons} = X^\mathrm{cons} \times X^\mathrm{cons}.$$
\begin{definition}\cite[Section 0060]{stacks-project}
    Let $X$ be a topological space. 
    \begin{enumerate}[label=(\arabic*)]
        \item If $x, x'\in X$ and $x \in \overline{\{x'\}}$, then we say that $x$ is a \emph{specialization} of $x'$. 
        \item A subset $E \subseteq X$ is \emph{stable under specialization} if for all $x' \in E$ and every specialization $x$ of $x'$ we have $x \in E$. 
    \end{enumerate}
\end{definition}
\begin{lemma}\cite[Section 08YF]{stacks-project}\label{closed_in_patch_and_stable_under_specialization_is_closed}
    Let $X$ be a spectral space. If a subset $E \subseteq X$ is closed in the patch topology and stable under specialization, then $E$ is closed in the spectral topology.
\end{lemma}
 
A useful tool for proving spectrality is a criterion involving ultrafilters, developed by Finocchiaro \cite{Finocchiaro03042014}. Before stating this criterion, we briefly recall the relevant notions of filters and ultrafilters on a set. 
\begin{definition}
    Given a set $S$ and a nonempty collection $\mathscr{F}$ of subsets of $S$, we say that $\mathscr{F}$ is a \emph{filter} on $S$ if the following properties hold: 
    \begin{enumerate}[label=(\arabic*)]
    \item $\varnothing \not\in\mathscr{F}$, 
    \item if $A, B \in \mathscr{F}$, then $A \cap B \in \mathscr{F}$, 
    \item if $A \in \mathscr{F}$ and $A \subseteq B \subseteq S$, then $B \in \mathscr{F}$. 
    \end{enumerate}
\end{definition}
\noindent A filter $\mathscr{F}$ is an \emph{ultrafilter} on $S$ if $\mathscr{F}$ is a maximal element (under inclusion) in the set of all filters on $S$. Equivalently, a filter $\mathscr{F}$ is an ultrafilter on $S$ if for each $A \subseteq S$, either $A \in\mathscr{F}$ or $S\setminus A \in\mathscr{F}$. We shall denote an ultrafilter by $\mathscr{U}$. 

\begin{theorem}\cite[Corollary 3.3]{Finocchiaro03042014}\label{ultrafilter_criterion}
    A topological space $X$ is spectral if and only if it satisfies the $T_0$-axiom and there is a subbasis $\mathbb{S}$ of $X$ such that 
    \begin{equation*}
        X_{\mathbb{S}}(\mathscr{U}) = \{x \in X \mid \forall S\in\mathbb{S} \hspace{3pt} x\in S \iff S \in\mathscr{U}\} \neq \varnothing
    \end{equation*}
    for any ultrafilter $\mathscr{U}$ on $X$. 
\end{theorem}

Let us recall that given a set $S$, the power set $\mathcal{P}(S)$ endowed with the \emph{hull-kernel topology} is a spectral space. For an arbitrary subset $E \subseteq S$, define
\begin{equation*}
    D(E) := \{A \in \mathcal{P}(S) \mid E \nsubseteq A\}, \qquad V(E) := \{A \in \mathcal{P}(S) \mid E \subseteq A\}. 
\end{equation*}
Thus, $V(E) = \mathcal{P}(S) \setminus D(E)$. The family $\{D(F)\}_{F\in\mathcal{P}_\mathrm{fin}(S)}$ forms an open subbasis of the hull-kernel topology on $\mathcal{P}(S)$. Moreover, each set of this family is quasi-compact in $\mathcal{P}(S)$. For a singleton $\{x\}$ we write $D(x)$ and $V(x)$ instead of $D(\{x\})$ and $V(\{x\})$. 

The \emph{specialization preorder} on a topological space $X$ is defined by 
\begin{equation*}
    A \leqslant B \iff B \in \overline{\{A\}}. 
\end{equation*}
It is well known that in the hull-kernel topology on $\mathcal{P}(S)$ this preorder coincides with set-theoretic inclusion. Indeed, let $A, B \in \mathcal{P}(S)$ and observe that 
\begin{equation*}\begin{aligned}
    B \in \overline{\{A\}} &\iff \forall F\in\mathcal{P}_\mathrm{fin}(S) \hspace{3pt} (F \subseteq A \implies B \in V(F)) \\ &\iff \forall F\in\mathcal{P}_\mathrm{fin}(S) \hspace{3pt} (F \subseteq A \implies F \subseteq B) \\ &\iff A \subseteq B. 
\end{aligned}\end{equation*}
Consequently, the specialization preorder on every subspace $X \subseteq\mathcal{P}(S)$ considered in the rest of the article is set-theoretic inclusion. 
\section{Spectral spaces arising from algebraic closure operators on semilattice-ordered semigroups}\label{spectral_spaces_arising_from_closures}
We begin by considering the family of all subsets that are closed with respect to a given closure operator $\mathrm{cl}$ on a semilattice-ordered semigroup $S$. Namely, we set $$X := \{A\in\mathcal{P}(S) \mid A^\mathrm{cl} = A\}$$ and endow it with the subspace topology induced by the hull-kernel topology on $\mathcal{P}(S)$. For every subset $E \subseteq S$, we write $$D_X(E) := D(E) \cap X, \qquad V_X(E) := V(E) \cap X.$$

Let us recall that quasi-compactness is not in general a hereditary property: if $X_1$ is quasi-compact and $X_2 \subseteq X_1$, then $X_2$ need not be quasi-compact. For this reason, even though each $D(F)$ is quasi-compact in $\mathcal{P}(S)$, it does not follow that the set $D_X(F)$ is quasi-compact in the induced topology on $X$. Lemma \ref{quasi_compactness_of_subbasic_opens} establishes that this property holds whenever the closure operator is algebraic.  
\begin{lemma}\label{quasi_compactness_of_subbasic_opens}
    Let $S$ be a semilattice-ordered semigroup and let $\mathrm{cl}$ be an algebraic closure operator on $S$. For every finite subset $F \subseteq S$, the set $D_X(F)$ is quasi-compact in $X$. 
\end{lemma}
\begin{proof}
    Fix a finite subset $F \subseteq S$. Since $X$ carries the subspace topology induced by the hull-kernel topology on $\mathcal{P}(S)$, the family of sets $$\{D_X(G)\}_{G\in\mathcal{P}_\mathrm{fin}(S)}$$ is an open subbasis of $X$. By Alexander's subbasis theorem, it is enough to show that every cover of $D_X(F)$ by subbasic open sets admits a finite subcover. Suppose then that we have such a cover, i.e.
    \begin{equation}\label{cover_by_subbasic_opens}
        D_X(F) \subseteq \bigcup_{i\in I} D_X(F_i),  
    \end{equation}
    where $F_i$ is a finite subset of $S$ for any $i\in I$. 
    Assume that no finite subcover exists. Then for every finite set $J \subseteq I$, $D_X(F) \nsubseteq \bigcup_{j \in J}D_X(F_j)$ and hence 
    \begin{equation*}
        D_X(F) \setminus\bigcup_{j \in J} D_X(F_j) = D_X(F) \cap \bigcap_{j \in J} V_X(F_j) \neq \varnothing. 
    \end{equation*}
    So for each finite set $J \subseteq I$, we can choose an element $C_J \in X$ such that $F \nsubseteq C_J$ and $\bigcup_{j \in J} F_j \subseteq C_J$. Since $C_J$ is $\mathrm{cl}$-closed and contains $\bigcup_{j\in J}F_j$, it must also contain $(\bigcup_{j \in J} F_j)^{\mathrm{cl}}$. Therefore, if $F \subseteq (\bigcup_{j \in J} F_j)^{\mathrm{cl}}$, then $F\subseteq C_J$, contradicting $F \nsubseteq C_J$. So for every finite set $J \subseteq I$ we have 
    \begin{equation}\label{F_is_not_contained_in_cl_of_finite_sum} 
    F \nsubseteq \Big(\bigcup_{j \in J} F_j\Big)^{\mathrm{cl}}.
    \end{equation}
    
    Now consider the set $A := \bigcup_{i \in I} F_i$. If $F \subseteq A^\mathrm{cl}$, then for each $f \in F$ there exists some finite set $B_f \subseteq A$ such that $f \in B_f^\mathrm{cl}$ (by algebraicity of $\mathrm{cl}$). Let $B := \bigcup_{f\in F}B_f \subseteq A$. Then the set $B$ is finite and hence there exists a finite set $K \subseteq I$ such that $B \subseteq \bigcup_{k \in K} F_k$. We get $$F \subseteq B^\mathrm{cl} \subseteq \Big(\bigcup_{k \in K} F_k\Big)^\mathrm{cl},$$ 
    which leads to a contradiction with \eqref{F_is_not_contained_in_cl_of_finite_sum}. Thus, $F \nsubseteq A^\mathrm{cl}$ and $A^\mathrm{cl} \in D_X(F)$. It is easy to see that also $A^\mathrm{cl} \in \bigcap_{i \in I} V_X(F_i)$. Therefore, 
    \begin{equation*}
        A^\mathrm{cl} \in D_X(F) \cap \bigcap_{i \in I} V_X(F_i) = D_X(F) \setminus \bigcup_{i \in I} D_X(F_i). 
    \end{equation*}
    This is a contradiction with \eqref{cover_by_subbasic_opens}. Thus, every cover of $D_X(F)$ by subbasic open sets must have a finite subcover. By Alexander's subbasis theorem, $D_X(F)$ is quasi-compact. 
\end{proof}
The next theorem is motivated by several earlier spectral space constructions arising from algebraic closure operators. Jun, Ray and Tolliver proved that for an algebraic closure operator $\mathrm{cl}$ the set of $\mathrm{cl}$-closed ideals of a semiring forms a spectral space \cite[Proposition~3.9]{JunRayTolliver}. Related results were obtained for the space of $\mathrm{cl}$-closed submodules of a given module over a commutative ring \cite[Proposition 3.4]{FinocchiaroFontanaSpiritoNullstellensatz}, as well as for the space of stable semistar operations of finite type on an integral domain \cite[Theorem 4.6]{FinocchiaroFontanaSpirito}. Banerjee obtained analogous results in the setting of abelian categories, where he constructed spectral spaces of invariants of closure operators acting on subobjects of a given object \cite[Proposition 3.3]{Banerjee}. Our argument follows the same general philosophy in the context of semilattice-ordered semigroups. In particular, the implication from algebraicity of $\mathrm{cl}$ to spectrality of the associated space $X$ is standard and we include it for completeness. 

\begin{theorem}\label{construction_of_spectral_space}
    Let $S$ be a semilattice-ordered semigroup and let $\mathrm{cl}$ be a closure operator on $S$. Then the space $X = \{A \in\mathcal{P}(S) \mid A^\mathrm{cl} = A\}$,
    endowed with the subspace topology induced by the hull-kernel topology on $\mathcal{P}(S)$, is a spectral space and a retrocompact subset of $\mathcal{P}(S)$ if and only if $\mathrm{cl}$ is algebraic.
\end{theorem}
\begin{proof}
    Assume first that $X$ is a spectral space and a retrocompact subset of $\mathcal{P}(S)$. To show that $\mathrm{cl}$ is algebraic, fix $A \subseteq S$. The inclusion $\bigcup\{F^\mathrm{cl} \mid F \in \mathcal{P}_\mathrm{fin}(A)\} \subseteq A^\mathrm{cl}$ follows from order-preservation, so it suffices to prove the reverse inclusion. Let $x\in A^\mathrm{cl}$. Then $V_X(A) \subseteq V_X(x)$ and thus $D_X(x) \subseteq D_X(A)$. It is easy to observe that $D_X(A) = \bigcup_{a \in A}D_X(a)$, so $\{D_X(a)\}_{a \in A}$ is an open cover of $D_X(x)$. Now $D(x)$ is a quasi-compact open subset of $\mathcal{P}(S)$ and $D_X(x) = D(x) \cap X$ is also quasi-compact, since $X$ is retrocompact in $\mathcal{P}(S)$. Hence there exists a finite subset $F \subseteq A$ such that 
    \begin{equation*}
        D_X(x) \subseteq \bigcup_{a \in F} D_X(a) = D_X(F), 
    \end{equation*}
    which yields $V_X(F) \subseteq V_X(x)$. Since $F^\mathrm{cl} \in X$ and $F \subseteq F^\mathrm{cl}$, we have $F^\mathrm{cl} \in V_X(F) \subseteq V_X(x)$, so $x \in F^\mathrm{cl}$. Therefore, every $x \in A^\mathrm{cl}$ belongs to $F^\mathrm{cl}$ for some finite subset $F \subseteq A$ and hence $A^\mathrm{cl} \subseteq \bigcup\{F^\mathrm{cl} \mid F \in \mathcal{P}_\mathrm{fin}(A)\}$. Thus, $\mathrm{cl}$ is algebraic. 
    
    Conversely, assume that $\mathrm{cl}$ is algebraic. Since $\mathcal{P}(S)$ with the hull-kernel topology is $T_0$ and $X$ carries the subspace topology, the space $X$ is also $T_0$. Now we show that $X$ is spectral by applying the ultrafilter criterion (Theorem \ref{ultrafilter_criterion}). Consider the subbasis $$\mathbb{S} := \{D_X(F) \mid F \in\mathcal{P}_\mathrm{fin}(S)\}.$$ Let $\mathscr{U}$ be an ultrafilter on $X$ and define $$A_\mathscr{U} := \{x \in S \mid V_X(x) \in \mathscr{U}\}.$$ We claim that $A_\mathscr{U} \in X_\mathbb{S}(\mathscr{U})$. We first show that $A_\mathscr{U}$ is $\mathrm{cl}$-closed. Let $x \in A_\mathscr{U}^\mathrm{cl}$. Since $\mathrm{cl}$ is algebraic, there exists a finite subset $F \subseteq A_\mathscr{U}$ such that $x \in F^\mathrm{cl}$. By definition of $A_\mathscr{U}$, $V_X(f) \in \mathscr{U}$ for every $f \in F$ and hence $$V_X(F) = \bigcap_{f \in F} V_X(f) \in \mathscr{U}.$$ Since every $\mathrm{cl}$-closed set containing $F$ also contains $x$, we have $V_X(F) \subseteq V_X(x)$. Hence $V_X(x) \in \mathscr{U}$, so $x \in A_\mathscr{U}$ and in consequence $A_\mathscr{U}^\mathrm{cl} \subseteq A_\mathscr{U}$. The reverse inclusion follows from the extension property of $\mathrm{cl}$. Thus, $A_\mathscr{U} \in X$. 
    
    Now let $F$ be a finite subset of $S$. Since $D_X(F) = \bigcup_{f \in F} D_X(f)$ and $\mathscr{U}$ is an ultrafilter, we have
    \begin{align*} 
        A_\mathscr{U} \in D_X(F) &\iff \exists f \in F \text{ such that } f \not\in A_\mathscr{U} \\ &\iff \exists f \in F \text{ such that } V_X(f) \not\in\mathscr{U} \\ &\iff \exists f \in F \text{ such that } D_X(f) \in \mathscr{U} \\ &\iff  
        D_X(F) \in \mathscr{U}. 
    \end{align*} 
    Hence $A_\mathscr{U}\in X_\mathbb{S}(\mathscr{U})$, and therefore $X$ is spectral.
    
    It remains to prove that $X$ is a retrocompact subset of $\mathcal{P}(S)$. Now let $U$ be a quasi-compact open subset of $\mathcal{P}(S)$. Then $$U = \bigcup_{k = 1}^n \bigcap_{l = 1}^{m_k} D(F_{kl})$$ for some finite sets $F_{kl} \subseteq S$. We get 
    \begin{equation*}
        U \cap X = \bigcup_{k = 1}^n \bigcap_{l = 1}^{m_k} \big(D(F_{kl}) \cap X\big) = \bigcup_{k = 1}^n\bigcap_{l = 1}^{m_k} D_X(F_{kl}). 
    \end{equation*}
    By Lemma \ref{quasi_compactness_of_subbasic_opens}, each $D_X(F_{kl})$ is quasi-compact in $X$. Since $X$ is spectral, finite intersections of quasi-compact open subsets of $X$ are again quasi-compact and open. Hence for every $k \in \{1, \ldots, n\}$ the set $\bigcap_{l = 1}^{m_k} D_X(F_{kl})$ is quasi-compact, and therefore $U \cap X$ is quasi-compact as a finite union of quasi-compact subsets. Thus, $X$ is a retrocompact subset of $\mathcal{P}(S)$. 
\end{proof}
Our next example shows that the algebraicity assumption in Theorem \ref{construction_of_spectral_space} is essential: if the closure operator is not algebraic, then the corresponding space $X$ may not be spectral. 
\begin{example}
    Let $S$ be the real unit interval, i.e. $S = [0,1]$, with the semigroup operations $x\cdot y = x \lor y := \min(x,y)$. A map $\mathrm{cl} : \mathcal{P}(S) \rightarrow \mathcal{P}(S)$ defined by
    \begin{equation*}
    A^{\mathrm{cl}} := 
    \begin{cases}
        \varnothing, & \text{for } A = \varnothing, \\
        [0, \sup(A)], & \text{for } A \neq \varnothing
    \end{cases}
    \end{equation*}
    is a closure operator on $S$. We claim that it is not algebraic. Indeed, if we take \begin{equation*}
        A := \bigg\{1 - \frac{1}{n} \hspace{3pt} \bigg| \hspace{3pt} n \in \mathbb{N}_{+}\bigg\},
    \end{equation*}
    then $\sup(A) = 1$, so $A^{\mathrm{cl}} = [0,1]$. On the other hand, 
    \begin{equation*}
        \bigcup\{F^{\mathrm{cl}} \mid F \in \mathcal{P}_\mathrm{fin}(A)\} = \bigcup_{n \in\mathbb{N}_+} \bigg[0, 1 - \frac{1}{n}\bigg] = [0, 1), 
    \end{equation*}
    hence $\mathrm{cl}$ fails algebraicity. 
    
    Let $X$ be defined as in Theorem \ref{construction_of_spectral_space} and let $F \neq \varnothing$ be a finite subset of $S$. Then \begin{equation*}
        X = \{\varnothing\} \cup \{[0,x] \mid x \in [0,1]\}
    \end{equation*}
    and 
    \begin{equation*}
        D_X(F) = \{\varnothing\} \cup \{[0, x] \mid 0 \leqslant x < \sup(F)\}. 
    \end{equation*}
    Hence an open subbasis of $X$ is given by the sets of the form $$D_t := \{\varnothing\} \cup \{[0,x] \mid 0 \leqslant x < t\}$$ for $t \in [0, 1].$ We claim that every nonempty proper open subset of $X$ is of the form $D_t$ for some $t\in[0,1]$. Indeed, let $U \subsetneq X$ be an open set. Since the family $\{D_t \mid t \in[0,1]\}$ is a subbasis, $U$ is a union of finite intersections of sets of the form $D_t$. But for all $s, t \in [0,1]$ we have $D_s \cap D_t = D_{\min(s, t)}$, and therefore every finite intersection of such sets is again of the form $D_t$. Hence for some family $\{t_i\}_{i \in I} \subseteq [0,1]$ we have $$U = \bigcup_{i\in I} D_{t_i} = D_{\sup_{i \in I}t_i}.$$
    
    Let us recall that a spectral space must have a basis of quasi-compact open sets. But for every $t > 0$ the set $D_t$ is not quasi-compact. Indeed, let $t_n$ be an increasing sequence such that $t_n \rightarrow t$ and $t_n < t$ for every natural number $n$. Then $D_t = \bigcup_{n \in \mathbb{N}} D_{t_n}$ and no finite subfamily covers $D_t$ because if $n_1, \ldots, n_k$ are fixed, then 
    \begin{equation*}
        D_{t_{n_1}} \cup \ldots \cup D_{t_{n_k}} = D_{\max(t_{n_1}, \ldots, t_{n_k})} \subsetneq D_t. 
    \end{equation*}
    Thus, $X$ is not spectral. 
\end{example}

\section{Spectrality of induced operations}\label{topological_semigroups_whose_underlying_space_is_spectral}
Having constructed a spectral space from an algebraic closure operator, the next natural question is whether the semigroup structure of $S$ can be transported to that space in a topologically meaningful way.  
Throughout this section, we use the same notation as before and additionally we define an operation $\star : X \times X \rightarrow X$ by 
\begin{equation*}
    A \star B = (AB)^\mathrm{cl}. 
\end{equation*}
\subsection{General case}
\begin{definition}
    Let $S$ be a semilattice-ordered semigroup. A closure operator $\mathrm{cl}$ on $S$ is said to be \emph{multiplicative}\footnote{The terminology "multiplicative closure operator" is consistent with earlier usage in semilattice theory (see e.g. \cite{Varlet_1973, Cornish_1974}). We also note that $\mathcal{P}(S)$, equipped with arbitrary unions and setwise multiplication, is a possibly nonunital quantale, and multiplicative closure operators on $\mathcal{P}(S)$ correspond to quantic nuclei. We nevertheless retain the present terminology, because it is more natural in the semigroup-theoretic context. Moreover, many of the results below are proved for arbitrary algebraic closure operators and do not require multiplicativity.} if $(AB)^\mathrm{cl} = (A^\mathrm{cl}B^\mathrm{cl})^\mathrm{cl}$ for all $A, B \subseteq S$. 
\end{definition}
\noindent Note that the inclusion $(AB)^\mathrm{cl} \subseteq (A^\mathrm{cl}B^\mathrm{cl})^\mathrm{cl}$ always holds.  
\begin{prop}\label{the_inverse_image_of_subbasic_closed_set_open_in_patch_topology}
    Let $S$ be a semilattice-ordered semigroup and let $\mathrm{cl}$ be a closure operator on $S$. 
    \begin{enumerate}[label=(\arabic*)]
    \item If $\mathrm{cl}$ is algebraic, then for every finite subset $F \subseteq S$, \begin{equation}\label{the_form_of_the_inverse_image_of_closed_subbasic_set}
        \star^{-1}\big(V_X(F)\big) = \bigcup_{\substack{G, H\in\mathcal{P}_\mathrm{fin}(S) \\ F \subseteq (GH)^\mathrm{cl}}}\big(V_X(G) \times V_X(H)\big).
    \end{equation}
    In particular, 
    the set $\star^{-1}(V_X(F))$ is open with respect to the patch topology on $X \times X$. 
    \item If $\mathrm{cl}$ is multiplicative, then $(X, \star)$ is a semigroup. 
    \end{enumerate}
\end{prop}
\begin{proof}
    $(1)$ Fix a finite subset $F \subseteq S$. Assume first that $(A, B)\in X \times X$ belongs to the right-hand side of \eqref{the_form_of_the_inverse_image_of_closed_subbasic_set}. Then there exist finite subsets $G, H\subseteq S$ such that $G \subseteq A$ and $H \subseteq B$, and $F \subseteq (GH)^\mathrm{cl}$. Since $GH \subseteq AB$, order-preservation of $\mathrm{cl}$ gives $$F \subseteq (GH)^\mathrm{cl} \subseteq (AB)^\mathrm{cl} = A\star B.$$ Thus, $(A, B)\in \star^{-1}(V_X(F))$. 

    Conversely, let $(A, B)\in \star^{-1}(V_X(F))$ and in consequence $F \subseteq (AB)^\mathrm{cl}$. Since $\mathrm{cl}$ is algebraic and $F$ is finite, there exists a finite subset $K\subseteq AB$ such that $F\subseteq K^\mathrm{cl}$. Thus, $K = \{a_1b_1, \ldots, a_nb_n\}$ for some elements $a_1, \ldots, a_n\in A$ and $b_1,\ldots, b_n\in B$. Now define $$G := \{a_1, \ldots, a_n\}, \qquad H := \{b_1, \ldots, b_n\}.$$ Then $K \subseteq GH$ and $(A , B)\in V_X(G) \times V_X(H)$. Therefore, $$F \subseteq K^\mathrm{cl} \subseteq (GH)^\mathrm{cl},$$ proving \eqref{the_form_of_the_inverse_image_of_closed_subbasic_set}. 
    
    For every finite subset $G\subseteq S$, since $D_X(G)$ is quasi-compact and open, its complement $V_X(G)$ is clopen in $X^\mathrm{cons}$. Hence $V_X(G)\times V_X(H)$ is open in $(X\times X)^\mathrm{cons}$ for all finite subsets $G, H\subseteq S$. Then by \eqref{the_form_of_the_inverse_image_of_closed_subbasic_set}, the set $\star^{-1}(V_X(F))$ is open in $(X\times X)^\mathrm{cons}$. 
    \\
    $(2)$ If $\mathrm{cl}$ is multiplicative, then for $A, B, C\in X$,
    \begin{equation*}
        (A \star B) \star C = \big((AB)^\mathrm{cl}C\big)^\mathrm{cl} = \big((AB)^\mathrm{cl}C^\mathrm{cl}\big)^\mathrm{cl} = (ABC)^\mathrm{cl}, 
    \end{equation*}
    and similarly $A \star (B \star C) = (ABC)^\mathrm{cl}$. Hence the operation $\star$ is associative. 
\end{proof}
The aim of the present section is to investigate when $\star$ is spectral with respect to the original spectral topology on $X$. The next proposition characterizes this property in terms of finite subsets of $S$. 
\begin{prop}\label{criterion_for_joint_spectrality}
    Let $S$ be a semilattice-ordered semigroup and let $\mathrm{cl}$ be an algebraic closure operator on $S$. The map $\star : X \times X \rightarrow X$ is spectral if and only if for every finite subset $F \subseteq S$ there exists a finite (possibly empty) family of pairs $(G_i, H_i) \in \mathcal{P}_\mathrm{fin}(S) \times \mathcal{P}_\mathrm{fin}(S)$, $i \in I$,
    such that for all $A, B\in X$, 
    \begin{equation*}
        F \subseteq A \star B \iff \exists i\in I \text{ such that } G_i \subseteq A \text{ and } H_i\subseteq B. 
    \end{equation*}
\end{prop}
\begin{proof}
    First we prove the implication from right to left. It suffices to show that $\star^{-1}(D_X(F))$ is a quasi-compact open set in $X \times X$ for every finite subset $F \subseteq S$. 
    
    Fix a finite subset $F \subseteq S$. There exists a finite (possibly empty) family of pairs $(G_i, H_i)\in\mathcal{P}_\mathrm{fin}(S) \times \mathcal{P}_\mathrm{fin}(S)$, $i \in I$, such that $$\star^{-1}\big(V_X(F)\big) = \{(A, B) \in X \times X \mid F \subseteq A \star B\} = \bigcup_{i \in I} \big(V_X(G_i) \times V_X(H_i)\big),$$ and then $$\star^{-1}\big(D_X(F)\big) = (X \times X)\setminus\star^{-1}\big(V_X(F)\big) = \bigcap_{i \in I} \Big(\big(X \times X\big) \setminus \big(V_X(G_i) \times V_X(H_i)\big)\Big).$$ If $I=\varnothing$, then $\star^{-1}(V_X(F)) = \varnothing$, hence $\star^{-1}(D_X(F)) = X\times X$, which is quasi-compact and open. Assume therefore that $I \neq \varnothing$. For every $i\in I$ we have $$\big(X \times X)\setminus\big(V_X(G_i) \times V_X(H_i)\big) = \big(D_X(G_i) \times X\big) \cup \big(X \times D_X(H_i)\big),$$
    which is quasi-compact and open, since the sets $D_X(G_i), D_X(H_i)$ are quasi-compact and open in $X$. In a spectral space any finite intersection of quasi-compact open sets is again a quasi-compact open set. Thus, the set $\star^{-1}(D_X(F))$ is quasi-compact and open in $X \times X$, as required.  
    
    Conversely, assume that the map $\star: X \times X \to X$ is spectral. Fix a finite subset $F \subseteq S$. Then $\star^{-1}(D_X(F))$ is quasi-compact and open in $X \times X$. It is therefore clopen in $(X \times X)^\mathrm{cons}$ and its complement $\star^{-1}(V_X(F))$ is compact in the patch topology. By Proposition \ref{the_inverse_image_of_subbasic_closed_set_open_in_patch_topology}(1), we have $$\star^{-1}\big(V_X(F)\big) = \bigcup_{\substack{G, H\in\mathcal{P}_\mathrm{fin}(S) \\ F \subseteq (GH)^\mathrm{cl}}}\big(V_X(G) \times V_X(H)\big).$$ Using compactness, we get finitely many pairs $(G_i, H_i)\in\mathcal{P}_\mathrm{fin}(S)\times\mathcal{P}_\mathrm{fin}(S)$, $i\in I$, such that $$\star^{-1}\big(V_X(F)\big) = \bigcup_{i\in I} \big(V_X(G_i) \times V_X(H_i)\big).$$ This gives the desired equivalence.  
\end{proof}

Our next example illustrates how Proposition \ref{criterion_for_joint_spectrality} can be applied in practice; more specifically, how one can verify that the condition appearing in that proposition does not hold.  
\begin{example}
Let $S := \mathbb{Z}$ with $$x \cdot y := x + y, \qquad x \lor y := \max(x, y).$$ Then $(S, \cdot, \lor)$ is a semilattice-ordered semigroup. Take the closure operator $\mathrm{cl} = \mathrm{id}_{\mathcal{P}(S)}$. It is algebraic, so we have  
\begin{equation*}\begin{aligned}
    X = \{A \in \mathcal{P}(S) \mid A^\mathrm{cl} = A\} = \mathcal{P}(\mathbb{Z})     
\end{aligned}\end{equation*}
and 
\begin{equation*}
    \star : X \times X \rightarrow X, \qquad A \star B = A + B = \{a + b \mid a \in A, \hspace{3pt} b \in B\}.
\end{equation*}
Using the criterion from Proposition \ref{criterion_for_joint_spectrality}, we show that the map $\star$ is not spectral. 

Assume towards a contradiction that $\star$ is spectral. Let $F := \{0\}$. Then there exist finitely many pairs of subsets $(G_1, H_1), \ldots, (G_n, H_n) \in \mathcal{P}_\mathrm{fin}(\mathbb{Z}) \times \mathcal{P}_\mathrm{fin}(\mathbb{Z})$ such that for all $A, B \subseteq \mathbb{Z}$, $$0 \in A + B \iff \exists i\in\{1,\ldots,n\} \text{ such that } G_i \subseteq A \text{ and } H_i \subseteq B.$$
Set $K := \bigcup_{i = 1}^n (G_i \cup (-H_i))$, where $-H_i := \{-h \mid h \in H_i\}$ for all $i \in \{1, \ldots, n\}$. The set $K$ is finite, so we can choose $m \in \mathbb{Z}\setminus K$. Put $A := \{m\}$ and $B := \{-m\}$. Then $0 \in A + B$ and hence $G_i \subseteq \{m\}$ and $H_i \subseteq \{-m\}$ for some $i\in\{1, \ldots, n\}$. At least one of the sets $G_i, H_i$ is nonempty (otherwise $0 \in A + B$ for all $A, B\subseteq\mathbb{Z}$). Thus, $m\in G_i$ or $-m\in H_i$, which implies $m \in K$, a contradiction. Therefore, $\star$ is not spectral. 
\end{example}

Besides the map $\star : X \times X \rightarrow X$, it is natural to consider the right and left translations obtained by fixing one argument. The next proposition shows that their spectrality can also be characterized by a condition analogous to that of Proposition \ref{criterion_for_joint_spectrality}. Since the proof follows a similar pattern, we omit it. 

\begin{prop}\label{criterion_for_separate_spectrality}
    Let $S$ be a semilattice-ordered semigroup and let $\mathrm{cl}$ be an algebraic closure operator on $S$. Let $B\in X$. The right translation $R_B : X \rightarrow X$ given by $R_B(A) = A \star B$ is spectral if and only if for every finite subset $F \subseteq S$ there exists a finite (possibly empty) family $G_i\in\mathcal{P}_\mathrm{fin}(S)$, $i\in I$,
    such that for all $A \in X$, 
    \begin{equation*}
        F \subseteq A \star B \iff \exists i\in I \text{ such that } G_i\subseteq A. 
    \end{equation*}
\end{prop}
\noindent The analogous equivalence holds for each left translation $L_A : X \rightarrow X$, $L_A(B) = A \star B$. 

The next corollary shows that spectrality of $\star$ implies spectrality of all right and left translations, as one would expect. The converse does not hold in general. A~counterexample will be given in Example \ref{separate_spectrality_does_not_imply_joint}.
\begin{corollary}\label{spectrality_of_star_implies_spectrality_of_all_translations}
    Let $S$ be a semilattice-ordered semigroup and let $\mathrm{cl}$ be an algebraic closure operator on $S$. If the map $\star : X \times X \rightarrow X$ is spectral, then the translations $L_A, R_B : X \rightarrow X$ are spectral for all $A, B \in X$. 
\end{corollary}
\begin{proof}
        Let $B \in X$ and let $F$ be a finite subset of $S$. By Proposition \ref{criterion_for_joint_spectrality}, there exist finitely many pairs $(G_i, H_i)\in\mathcal{P}_\mathrm{fin}(S)\times\mathcal{P}_\mathrm{fin}(S)$, $i\in I$, such that for all $A, B' \in X$, $$F \subseteq A \star B' \iff \exists i\in I\text{ such that } G_i \subseteq A\text{ and } H_i\subseteq B'.$$ Now define $$I_B := \{i\in I \mid H_i \subseteq B\}$$ and then $$F \subseteq A \star B \iff \exists i\in I_B\text{ such that } G_i\subseteq A.$$ By Proposition \ref{criterion_for_separate_spectrality}, the right translation $R_B: X \to X$ is spectral. The proof for left translations is analogous. 
\end{proof}

\begin{definition}\cite{Herrmann}
    Let $S$ be a semilattice-ordered semigroup and let $\mathrm{cl}$ be a closure operator on $S$. A subset $A \in X$ is said to be \emph{finitely generated} if there exists a finite set $F \in\mathcal{P}_\mathrm{fin}(S)$ such that $A = F^\mathrm{cl}$.   
\end{definition}
We shall denote by $X_\mathrm{fin}$ the family of all finitely generated $\mathrm{cl}$-closed subsets of $S$, i.e. 
\begin{equation*}
    X_\mathrm{fin} := \{F^\mathrm{cl} \mid F \in \mathcal{P}_\mathrm{fin}(S)\} \subseteq X. 
\end{equation*}
\begin{prop}\label{spectrality_of_star_in_terms_of_finitely_generated_subsets}
    Let $S$ be a semilattice-ordered semigroup and let $\mathrm{cl}$ be an algebraic closure operator on $S$. The map $\star : X \times X \rightarrow X$ is spectral if and only if for every finite subset $F \subseteq S$ there exists a finite (possibly empty) family of pairs $(G_i, H_i) \in X_\mathrm{fin} \times X_\mathrm{fin}$, $i \in I$, such that for all $A, B\in X_\mathrm{fin}$, 
    \begin{equation}\label{criterion_for_fin}
        F \subseteq A \star B \iff \exists i\in I\text{ such that } G_i \subseteq A\text{ and } H_i \subseteq B. 
    \end{equation}
\end{prop}
\begin{proof}
    The implication from left to right follows clearly from Proposition \ref{criterion_for_joint_spectrality}. Conversely, fix a finite subset $F \subseteq S$. Then there exists a finite (possibly empty) family of pairs $(G_i, H_i) \in X_\mathrm{fin} \times X_\mathrm{fin}$, $i \in I$, such that the condition \eqref{criterion_for_fin} holds. For every $i\in I$ we have $G_i = E_i^\mathrm{cl}$ and $H_i = K_i^\mathrm{cl}$ for some $E_i, K_i \in \mathcal{P}_\mathrm{fin}(S)$. We claim that for all $A, B \in X$, 
    \begin{equation*}
        F \subseteq A \star B \iff \exists i\in I \text{ such that } E_i \subseteq A \text{ and } K_i \subseteq B, 
    \end{equation*}
    which yields that $\star$ is spectral (by Proposition \ref{criterion_for_joint_spectrality}). 
    
    First assume that $F \subseteq A \star B$ with $A, B \in X$. Since $F \subseteq (AB)^\mathrm{cl}$, algebraicity yields a finite set $K \subseteq AB$ such that $F\subseteq K^\mathrm{cl}$. Choose finite sets $C\subseteq A$ and $D\subseteq B$ such that $K\subseteq CD$. Then $F \subseteq K^\mathrm{cl} \subseteq (CD)^\mathrm{cl}$. Now put $A' := C^\mathrm{cl}$ and $B' := D^\mathrm{cl}$. Since $A, B$ are $\mathrm{cl}$-closed and $C \subseteq A$, $D \subseteq B$, we have $A' \subseteq A$ and $B' \subseteq B$. Moreover, $$F \subseteq (CD)^\mathrm{cl} \subseteq (A'B')^\mathrm{cl} =A'\star B'.$$ Applying \eqref{criterion_for_fin} to $A', B' \in X_\mathrm{fin}$, we have $G_i \subseteq A'$ and $H_i \subseteq B'$ for some $i \in I$. Thus, $E_i \subseteq A$ and $K_i \subseteq B$, as required. 
    
    Now assume that $E_i \subseteq A$ and $K_i \subseteq B$ for some $i\in I$. Then $G_i = E_i^\mathrm{cl} \subseteq A$ and $H_i = K_i^\mathrm{cl} \subseteq B$, since $A, B$ are $\mathrm{cl}$-closed. Applying \eqref{criterion_for_fin} to $G_i, H_i\in X_\mathrm{fin}$, we have $F \subseteq G_i \star H_i \subseteq A \star B$. This completes the proof. 
\end{proof}
The next corollary brings together the preceding characterizations of spectrality of the map $\star$. 
\begin{corollary}\label{equivalent_conditions_on_spectrality}
    Let $S$ be a semilattice-ordered semigroup and let $\mathrm{cl}$ be an algebraic closure operator on $S$. Then the following conditions are equivalent: 
    \begin{enumerate}[label=(\arabic*)]
        \item the map $\star : X \times X \rightarrow X$ is spectral, 
        \item the map $\star: X^\mathrm{cons}\times X^\mathrm{cons}\to X^\mathrm{cons}$ is continuous, 
        \item for every finite subset $F \subseteq S$ there exists a finite (possibly empty) family of pairs $(G_i, H_i) \in \mathcal{P}_\mathrm{fin}(S) \times \mathcal{P}_\mathrm{fin}(S)$, $i \in I$, such that for all $A, B\in X$, 
        \begin{equation*}
            F \subseteq A \star B \iff \exists i \in I \text{ such that } G_i \subseteq A \text{ and } H_i\subseteq B, 
        \end{equation*}
        \item for every finite subset $F \subseteq S$ there exists a finite (possibly empty) family of pairs $(G_i, H_i) \in X_\mathrm{fin} \times X_\mathrm{fin}$, $i \in I$, such that for all $A, B\in X_\mathrm{fin}$, 
        \begin{equation*}
            F \subseteq A \star B \iff \exists i\in I \text{ such that } G_i \subseteq A\text{ and } H_i \subseteq B.
        \end{equation*}
    \end{enumerate}
\end{corollary}
\begin{proof}
    By Lemma \ref{spectrality_implies_continuity_in_the_patch_topology}, the condition $(1)$ implies $(2)$. The equivalence of $(1)$ and $(3)$ follows from Proposition \ref{criterion_for_joint_spectrality}, whereas the equivalence of $(1)$ and $(4)$ follows from the preceding proposition. Therefore, it remains to prove that the condition $(2)$ implies $(1)$.  

    Assume that the map $\star: X^\mathrm{cons}\times X^\mathrm{cons} \to X^\mathrm{cons}$ is continuous. It is enough to show that $\star^{-1}(D_X(F))$ is quasi-compact and open in $X \times X$ for every finite subset $F \subseteq S$. Fix such a set $F$. Since $D_X(F)$ is quasi-compact and open in $X$, it is clopen in $X^\mathrm{cons}$. Hence $\star^{-1}(D_X(F))$ is clopen in $X^\mathrm{cons}\times X^\mathrm{cons}$. Since this space is quasi-compact, $\star^{-1}(D_X(F))$ is quasi-compact with respect to the patch topology. As the patch topology is finer than the spectral topology, it is also quasi-compact in the spectral topology on $X \times X$. 

    It remains to prove that $\star^{-1}(D_X(F))$ is open in the spectral topology. Consider its complement $$\star^{-1}\big(V_X(F)\big) = \{(A, B)\in X \times X \mid F \subseteq A \star B\},$$ which is a closed set in $X^\mathrm{cons} \times X^\mathrm{cons}$. Moreover, it is stable under specialization. Indeed, if $(A, B) \in \star^{-1}(V_X(F))$ and $A \subseteq A'$, $B\subseteq B'$, then $F \subseteq A\star B \subseteq A'\star B'$. Thus, $(A', B')\in \star^{-1}(V_X(F))$. By Lemma \ref{closed_in_patch_and_stable_under_specialization_is_closed}, the set $\star^{-1}(V_X(F))$ is closed in the spectral topology on $X \times X$ and in consequence the set $\star^{-1}(D_X(F))$ is open. 
\end{proof}
\begin{definition}\cite[Definition 1.1]{schmitz}
    A preordered set $(P, \leqslant)$ is \emph{well-quasi-ordered} if every infinite sequence $(x_n)_{n \in\mathbb{N}}$ of elements of $P$ contains an increasing pair, i.e. there exist indices $m < n$ such that $x_m \leqslant x_n$. 
\end{definition}
The next proposition gives a useful sufficient condition for spectrality of the operation $\star$ in terms of the order structure of finitely generated closed subsets of $S$. 

\begin{prop}\label{well_quasi_order_implies_spectrality}
    Let $S$ be a semilattice-ordered semigroup and let $\mathrm{cl}$ be an algebraic closure operator on $S$. If the poset $(X_\mathrm{fin}, \subseteq)$ is well-quasi-ordered, then the map $\star : X \times X \rightarrow X$ is spectral. 
\end{prop}
\begin{proof}
    Fix a finite subset $F \subseteq S$ and define 
    \begin{equation*}
        U_F := \{(G, H)\in X_\mathrm{fin} \times X_\mathrm{fin} \mid F \subseteq G \star H\}. 
    \end{equation*}
    The map $\star$ is order-preserving in each variable and hence the set $U_F$ is upward closed with respect to the product order on $X_\mathrm{fin} \times X_\mathrm{fin}$. Since $(X_\mathrm{fin}, \subseteq~)$ is well-quasi-ordered, $X_\mathrm{fin} \times X_\mathrm{fin}$ is also well-quasi-ordered \cite[Lemma 1.8]{schmitz}. Thus, every upward closed subset of $X_\mathrm{fin} \times X_\mathrm{fin}$ can be written as an upward closure of its finitely many minimal elements \cite[Lemma 1.7]{schmitz}. In particular, the set $U_F$ can be written in such a form. Hence there exists a finite (possibly empty) family of pairs $(G_i, H_i) \in X_\mathrm{fin} \times X_\mathrm{fin}$, $i \in I$, such that for all $A, B\in X_\mathrm{fin}$, 
    \begin{equation*}
        F \subseteq A\star B \iff \exists i\in I \text{ such that } G_i \subseteq A \text{ and } H_i \subseteq B. 
    \end{equation*}
    So by Proposition \ref{spectrality_of_star_in_terms_of_finitely_generated_subsets}, the map $\star : X \times X \rightarrow X$ is spectral. 
\end{proof}
\begin{definition}
    Let $S$ be a semilattice-ordered semigroup. A closure operator $\mathrm{cl}$ on $S$ is said to be \emph{order-compatible} if $$x \leqslant y \implies \{x\}^\mathrm{cl} \subseteq \{y\}^\mathrm{cl}$$ for all $x, y \in S$. 
\end{definition}
\begin{prop}
    Let $S$ be a semilattice-ordered semigroup and let $\mathrm{cl}$ be an order-compatible closure operator on $S$. If the poset $(S, \leqslant)$ is well-quasi-ordered, then the poset $(X_\mathrm{fin}, \subseteq)$ is well-quasi-ordered. 
\end{prop}
\begin{proof}
    Consider the Hoare preorder $\sqsubseteq_\mathrm{H}$ on $\mathcal{P}_\mathrm{fin}(S)$ defined by $$F \sqsubseteq_\mathrm{H} G \iff \forall x\in F \hspace{3pt} \exists y\in G \text{ such that } x \leqslant y.$$ Let $(A_n)_{n\in\mathbb{N}}$ be an infinite sequence of elements of $X_\mathrm{fin}$. For each $n\in\mathbb{N}$ choose a finite subset $F_n\subseteq S$ such that $A_n = F_n^\mathrm{cl}$. Since $(S, \leqslant)$ is well-quasi-ordered, the set $\mathcal{P}_\mathrm{fin}(S)$ is well-quasi-ordered with respect to $\sqsubseteq_\mathrm{H}$ (see \cite[Exercise 1.13(4)]{schmitz}). Thus, there exist indices $m < n$ such that $F_m \sqsubseteq_\mathrm{H} F_n$. We claim that $A_m \subseteq A_n$. Let $x\in F_m$. By the definition of the Hoare preorder, there exists $y\in F_n$ such that $x \leqslant y$. Using order-compatibility of $\mathrm{cl}$, we obtain $$\{x\}^\mathrm{cl} \subseteq \{y\}^\mathrm{cl} \subseteq F_n^\mathrm{cl} = A_n.$$ In particular, $x\in A_n$ and hence $F_m \subseteq A_n$. Since $A_n$ is $\mathrm{cl}$-closed, we get $$A_m = F_m^\mathrm{cl} \subseteq A_n^\mathrm{cl} = A_n.$$ Thus, every infinite sequence of elements of $X_\mathrm{fin}$ contains an increasing pair. Therefore, $(X_\mathrm{fin}, \subseteq)$ is well-quasi-ordered.
\end{proof}
\begin{corollary}\label{algebraicity_order_compatibility_and_well_quasi_order_on_original_set_imply_spectrality_of_star}
    Let $S$ be a semilattice-ordered semigroup and let $\mathrm{cl}$ be an algebraic, order-compatible closure operator on $S$. If the poset $(S, \leqslant)$ is well-quasi-ordered, then the map $\star: X \times X \to X$ is spectral. 
\end{corollary}
The preceding results obtain spectrality of $\star$ from properties of the order of finitely generated closed subsets. There is also another situation in which spectrality of $\star$ can be recovered. Namely, when the induced semigroup $(X, \star)$ is a semilattice, spectrality of the translations already determines spectrality of the binary operation. 
\begin{prop}\label{separate_spectrality_implies_joint_for_semilattices}
    Let $S$ be a semilattice-ordered semigroup and let $\mathrm{cl}$ be an algebraic closure operator. Suppose that the semigroup $(X, \star)$ is a semilattice. Then the following conditions are equivalent: 
    \begin{enumerate}[label=(\arabic*)]
        \item the map $\star: X \times X \to X$ is spectral, 
        \item for every $B \in X$, the right translation $R_B: X \to X$ is spectral, 
        \item for every $A \in X$, the left translation $L_A: X \to X$ is spectral. 
    \end{enumerate}
\end{prop}
\begin{proof}
    By Corollary \ref{spectrality_of_star_implies_spectrality_of_all_translations}, the condition $(1)$ implies conditions $(2)$ and $(3)$. Since $\star$ is commutative, we have $L_A = R_A$ for every $A \in X$ and hence conditions $(2)$ and $(3)$ are equivalent. 

    Assume now that the condition $(2)$ holds. Consequently, every right translation $R_B: X^\mathrm{cons} \to X^\mathrm{cons}$ is continuous. By Theorem \ref{spectral_space_with_patch_top_is_compact}, $(X^\mathrm{cons}, \star)$ is a quasi-compact Hausdorff semitopological semilattice. By \cite[Theorem VII-4.8]{Gierz_Hofmann_Keimel_Lawson_Mislove_Scott_200}, it is in fact topological. Therefore, the map $\star: X^\mathrm{cons} \times X^\mathrm{cons} \to X^\mathrm{cons}$ is continuous and Corollary \ref{equivalent_conditions_on_spectrality} implies that $\star: X \times X \to X$ is spectral.
\end{proof}
\subsection{The identity closure operator}
We conclude this section by considering the identity closure operator. In this case $X = \mathcal{P}(S)$ and $A \star B = AB$, so spectrality of the induced operation can be expressed directly in terms of the multiplicative structure of $S$. For an element $s\in S$, the condition $s\in AB$ is equivalent to the existence of a decomposition $s = ab$, where $a\in A$ and $b\in B$. This leads naturally to the finite decomposition property. 
\begin{definition}\cite[Section 2]{DeneufchatelDuchamp}
    A semigroup $S$ has \emph{finite decomposition property} if the set $$\mathrm{Dec}(s) := \big\{(a, b) \in S \times S \mid ab = s\big\}$$ is finite for every $s \in S$. 
\end{definition}
Finite decomposition semigroups and analogous finiteness conditions on related structures have previously been studied in connection with convolution algebras and formal power series \cite{DeneufchatelDuchamp, FahrenbergJohansenStruthZiemianski}. Our next result provides a topological characterization of this property. 
\begin{theorem}\label{spectrality_is_equivalent_to_finite_decomposition_property}
    Let $S$ be a semilattice-ordered semigroup and let $X = \mathcal{P}(S)$. The map $\star : X \times X \to X$ is spectral if and only if $S$ has finite decomposition property. 
\end{theorem}
\begin{proof}
    Assume first that the map $\star: X \times X \to X$ is spectral and fix an element $s\in S$. Then $\star^{-1}(D_X(s))$ is quasi-compact and open, so it is clopen with respect to the patch topology on $X \times X$. Therefore, its complement $\star^{-1}(V_X(s))$ is quasi-compact in the patch topology on $X \times X$. Moreover, since $$\star^{-1}\big(V_X(s)\big) = \bigcup_{(a, b)\in\mathrm{Dec}(s)} V_X(a) \times V_X(b)$$ and all sets of the form $V_X(a) \times V_X(b)$ are clopen in $X^\mathrm{cons} \times X^\mathrm{cons}$, we may choose finitely many pairs $(a_1, b_1), \ldots, (a_n, b_n)$ for which 
    \begin{equation}\label{finite_subcover_of_inverse_image}
        \star^{-1}\big(V_X(s)\big) = \bigcup_{i = 1}^n V_X(a_i) \times V_X(b_i).
    \end{equation}
    Now take an element $(a, b)\in\mathrm{Dec}(s)$. Then we have $ab = s$ and in consequence $(\{a\}, \{b\})\in \star^{-1}(V_X(s))$. By \eqref{finite_subcover_of_inverse_image}, we get $a = a_i$ and $b = b_i$ for some $i\in\{1, \ldots, n\}$. Thus, the set $\mathrm{Dec}(s)$ is finite. 

    Conversely, assume that $S$ has finite decomposition property. For every $s \in S$ we have $$\star^{-1}\big(V_X(s)\big) = \bigcup_{(a,b)\in\mathrm{Dec}(s)}V_X(a) \times V_X(b).$$
    Therefore, $$\star^{-1}\big(D_X(s)\big) = \bigcap_{(a,b)\in\mathrm{Dec}(s)} \big((D_X(a) \times X) \cup (X \times D_X(b))\big),$$ which is quasi-compact and open. Since $D_X(F) = \bigcup_{s \in F} D_X(s)$ for every finite subset $F \subseteq S$, the map $\star$ is spectral. 
\end{proof}
The following two examples illustrate Theorem \ref{spectrality_is_equivalent_to_finite_decomposition_property} in commutative and noncommutative settings, respectively. 
\begin{example}
Let $R$ be a Dedekind domain and let $$S := \{I \subseteq R \mid I\text{ is a nonzero ideal of } R\}.$$ We equip $S$ with multiplication $$I\cdot J := \bigg\{\sum_{k = 1}^n a_kb_k \mid n\in\mathbb{N}_+, a_k\in I, b_k\in J\bigg\}$$ and join operation $$I \lor J := I + J = \{a + b \mid a \in I, b\in J\}.$$ Then $(S, \cdot, \lor)$ is a commutative semilattice-ordered semigroup. 

We claim that $S$ has finite decomposition property. Fix an element $I \in S$. If $I \neq R$, then by unique factorization of nonzero ideals in a Dedekind domain, we have $$I = \mathfrak{p}_1^{n_1}\ldots\mathfrak{p}_r^{n_r},$$ where $\mathfrak{p}_1, \ldots, \mathfrak{p}_r$ are distinct nonzero prime ideals and $n_1, \ldots, n_r$ are positive integers. 

Suppose that $I = JK$ for some nonzero ideals $J, K$. Then there are uniquely determined integers $k_1, \ldots, k_r$ such that $$J = \mathfrak{p}_1^{k_1} \ldots\mathfrak{p}_r^{k_r},\qquad K = \mathfrak{p}_1^{n_1 - k_1}\ldots\mathfrak{p}_r^{n_r - k_r},$$ and $0 \leqslant k_i \leqslant n_i$ for each $i\in\{1, \ldots, r\}$. Conversely, every such choice of $k_1, \ldots, k_r$ gives a decomposition $I = JK$. We conclude that $$|\mathrm{Dec}(I)| = \prod_{i = 1}^r (n_i + 1),$$ so the set $\mathrm{Dec}(I)$ is finite. If $I = R$, then $\mathrm{Dec}(R) = \{(R, R)\}$. Therefore, $S$ has finite decomposition property. It follows from Theorem \ref{spectrality_is_equivalent_to_finite_decomposition_property} that, for $X = \mathcal{P}(S)$, the induced multiplication $\star: X \times X \to X$ is spectral. 
\end{example}
\begin{example}
Let $(L, \lor)$ be a semilattice with at least two elements and put $S := \mathbb{N}\times L$. For $(n, x), (m, y) \in S$, define multiplication $$(n, x) \cdot (m, y) := (\max(n, m), x)$$ and join operation $$(n, x) \lor (m, y) := (\min(n, m), x\lor y).$$ Then $(S, \cdot, \lor)$ is a noncommutative semilattice-ordered semigroup. Moreover, if $a, b, c, d \in S$, then $$a^2 = a, \qquad (ab)(cd) = (ac)(bd),$$ so $(S, \cdot)$ is a mode and $(S, \cdot, \lor)$ is a noncommutative modal in the terminology of~\cite{PilitowskaZamojskaDzienio}.

Fix an element $(n, x)\in S$. Then $$\mathrm{Dec}(n , x) = \big\{\big((m, x), (k, y)\big)\in S \times S\mid\max(m, k) = n\big\}.$$ There are exactly $2n + 1$ ordered pairs $(m, k)\in\mathbb{N} \times \mathbb{N}$ satisfying $\max(m, k) = n$. Thus, we get $$|\mathrm{Dec}(n, x)| = (2n + 1)|L|.$$ Let $X = \mathcal{P}(S)$. By Theorem \ref{spectrality_is_equivalent_to_finite_decomposition_property}, the induced multiplication $\star: X \times X \to X$ is spectral if and only if $L$ is finite. 
\end{example}
For any subsets $A, B \subseteq S$ and an element $s \in S$, define sets $$\mathrm{LFac}_B(s) := \{a \in S \mid \exists b\in B\text{ such that } ab = s\}$$ and $$\mathrm{RFac}_A(s) := \{b \in S \mid \exists a\in A\text{ such that } ab = s\}.$$
The next proposition shows that spectrality of right and left translations can also be characterized by a condition analogous to that of Theorem \ref{spectrality_is_equivalent_to_finite_decomposition_property}. Since the proof follows a similar pattern, we omit it. 
\begin{prop}\label{spectrality_of_translations_are_equivalent_to_finitness_of_proper_factorizations_sets}
    Let $S$ be a semilattice-ordered semigroup and let $X = \mathcal{P}(S)$. 
    \begin{enumerate}[label=(\arabic*)]
        \item For a subset $B \subseteq S$, the right translation $R_B: X \to X$ is spectral if and only if $\mathrm{LFac}_B(s)$ is finite for every $s\in S$. 
        \item For a subset $A \subseteq S$, the left translation $L_A: X \to X$ is spectral if and only if $\mathrm{RFac}_A(s)$ is finite for every $s \in S$. 
    \end{enumerate}
\end{prop}
\begin{corollary}\label{spectrality_of_star_is_equivalent_to_spectrality_of_all_left_and_right_translations}
    Let $S$ be a semilattice-ordered semigroup and let $X = \mathcal{P}(S)$. Then the following conditions are equivalent: 
    \begin{enumerate}[label=(\arabic*)]
        \item the map $\star: X \times X \to X$ is spectral,
        \item for all $A, B \in X$, the translations $L_A, R_B: X \to X$ are spectral, 
        \item the right translation $R_S: X \to X$ and the left translation $L_S: X \to X$ are spectral. 
    \end{enumerate}
\end{corollary}
\begin{proof}
    The implication from $(1)$ to $(2)$ follows from Corollary \ref{spectrality_of_star_implies_spectrality_of_all_translations} and the implication from $(2)$ to $(3)$ is immediate. Assume that the translations $R_S, L_S: X\to X$ are spectral. Fix an element $s \in S$. By Proposition \ref{spectrality_of_translations_are_equivalent_to_finitness_of_proper_factorizations_sets}, the sets $\mathrm{LFac}_S(s)$ and $\mathrm{RFac}_S(s)$ are finite. Since we have $$\mathrm{Dec}(s) \subseteq \mathrm{LFac}_S(s)\times\mathrm{RFac}_S(s),$$ the set $\mathrm{Dec}(s)$ is also finite. Then by Theorem \ref{spectrality_is_equivalent_to_finite_decomposition_property}, the map $\star: X \times X \to X$ is spectral. 
\end{proof}
\begin{corollary}
    Let $S$ be a semilattice-ordered semigroup and let $X = \mathcal{P}(S)$. 
    \begin{enumerate}[label=(\arabic*)]
        \item If $S$ is commutative, then $\star: X \times X \to X$ is spectral if and only if  $R_S: X \to X$ is spectral.
        \item If $S$ is left cancellative, then $\star: X \times X \to X$ is spectral if and only if $R_S: X \to X$ is spectral.
        \item If $S$ is right cancellative, then $\star: X \times X \to X$ is spectral if and only if $L_S: X \to X$ is spectral. 
    \end{enumerate}
\end{corollary}
\begin{proof}
    $(1)$ In the commutative case we have $R_S = L_S$, so the conclusion follows from Corollary \ref{spectrality_of_star_is_equivalent_to_spectrality_of_all_left_and_right_translations}. \\
    $(2)$ Fix an element $s \in S$. In the left cancellative case, for each fixed $a \in S$ the equation $ab = s$ has at most one solution $b$. Thus, $|\mathrm{Dec}(s)| = |\mathrm{LFac}_S(s)|.$ Then the conclusion follows from Proposition \ref{spectrality_of_translations_are_equivalent_to_finitness_of_proper_factorizations_sets} and Theorem \ref{spectrality_is_equivalent_to_finite_decomposition_property}. \\
    $(3)$ In this case the proof is analogous to the proof of $(2)$. 
\end{proof}

\section{Closure operators on semilattice-ordered semigroups}\label{closures_on_semigroups}
In this section we discuss several closure operators naturally associated with semilattice-ordered semigroups. Our motivation is twofold. First, by Theorem \ref{construction_of_spectral_space}, every algebraic closure operator $\mathrm{cl}$ gives rise to a spectral space 
\begin{equation*}
    X = \{A \in \mathcal{P}(S) \mid A^\mathrm{cl} = A\},
\end{equation*}
endowed with the subspace topology induced by the hull-kernel topology on $\mathcal{P}(S)$. Thus, identifying natural algebraic closure operators provides concrete examples of the spectral spaces studied in Section \ref{spectral_spaces_arising_from_closures}. Second, for each such closure operator, the induced operation $A \star B = (AB)^\mathrm{cl}$ can be studied using the spectrality criteria developed in Section \ref{topological_semigroups_whose_underlying_space_is_spectral}.

\medskip \noindent 
\textbf{Some simple examples of closure operators on semilattice-ordered semigroups.} We begin by recalling a few classical closure operators from ordered algebra and universal algebra which will be used repeatedly below. 

For a subset $A\subseteq S$, the \emph{downward closure} 
\begin{equation*}
    A \mapsto \downarrow A := \{ x\in S \mid \exists a\in A \text{ such that } x \leqslant a\}
\end{equation*}
is the usual operator in a poset. If $A = \downarrow A$, then $A$ is called a \emph{downset}. The \emph{finite-join closure} 
\begin{equation*} 
    A\mapsto A^{\lor} := \{a_1 \lor \ldots \lor a_n \mid n \geqslant 1, a_i \in A\}
\end{equation*}
is standard in join-semilattice theory. Furthermore, 
\begin{equation*}
    A \mapsto \sqrt{A} := \{x \in S \mid \exists n\in\mathbb{N}_+ \text{ such that } x^n \in A\}
\end{equation*}
defines a closure operator on $S$ (see \cite[Section 1]{KollarSulka}). 
\begin{remark}
    In ring theory the same construction is most commonly used on the set of ideals of a ring. If $\mathcal{J}$ is an ideal of a ring $R$, then $$\sqrt{\mathcal{J}} := \{r \in R \mid \exists n\in\mathbb{N}_+ \text{ such that } r^n \in \mathcal{J}\}$$ is again an ideal called the \emph{radical} of $\mathcal{J}$ (see e.g. \cite[Section 1]{matsumura1980commutative}). 
\end{remark}
\begin{prop}\label{order_compatibility_is_equivalent_to_all_closed_subsets_being_down_sets}
    Let $S$ be a semilattice-ordered semigroup. A closure operator $\mathrm{cl}$ on $S$ is order-compatible if and only if every $\mathrm{cl}$-closed subset of $S$ is a downset.
\end{prop}
\begin{proof}
    Assume first that $\mathrm{cl}$ is order-compatible. Let $A$ be a $\mathrm{cl}$-closed subset of $S$. If $x \leqslant y$ and $y \in A$, then $$\{x\}^\mathrm{cl}\subseteq\{y\}^\mathrm{cl}\subseteq A^\mathrm{cl} = A.$$ Since $x\in\{x\}^\mathrm{cl}$, it follows that $x\in A$. Thus, $A$ is a downset. 

    Conversely, suppose that every $\mathrm{cl}$-closed subset of $S$ is a downset. Let $x, y\in S$ be such that $x \leqslant y$. Since $\{y\}^\mathrm{cl}$ is a $\mathrm{cl}$-closed subset of $S$ containing $y$, it is a downset and therefore $x \in \{y\}^\mathrm{cl}$. By order-preservation and idempotence of $\mathrm{cl}$, we get $$\{x\}^\mathrm{cl}\subseteq\big(\{y\}^\mathrm{cl}\big)^\mathrm{cl} = \{y\}^\mathrm{cl}.$$ Thus, the closure operator $\mathrm{cl}$ is order-compatible. 
\end{proof}
\subsection{Downward-join closure} 
We begin with the motivation for the construction considered below. In \cite[Theorem 3.7]{LESCOT20111782}, Lescot considered a $\mathbb{B}_1$-algebra $A$ and the smallest congruence $\mathcal{C}_\mathcal{J} \subseteq A \times A$, where $\mathcal{J}$ is an ideal of $A$, such that $$\forall x \in \mathcal{J} \hspace{3pt} (x, 0) \in \mathcal{C}_\mathcal{J}.$$ He proved that $$(x, y)\in\mathcal{C}_\mathcal{J} \iff \exists z \in \mathcal{J} \text{ such that } x + z = y + z,$$ and that the assignment $$\mathcal{J} \mapsto \mathcal{I}(\mathcal{C}_\mathcal{J}) := \{x \in A \mid (x, 0)\in\mathcal{C}_\mathcal{J}\} = \{x\in A \mid \exists z\in \mathcal{J} \text{ such that } x + z = z\}$$ defines a closure operator on the set of ideals of $A$. Recently, Jun, Ray and Tolliver recalled the same construction in the setting of additively idempotent semirings, introducing the $k$-closure operator \cite[Section 5]{JunRayTolliver}. 

In the present setting there is no distinguished additive identity and the join operation $\lor$ plays the role of the idempotent addition. Moreover, we adapt this construction to arbitrary subsets. Let $S$ be a semilattice-ordered semigroup and let $A$ be a subset of $S$. If $A = \varnothing$, let $\mathcal{C}_A$ be the equality relation on $S$. If $A \neq \varnothing$, we define a relation $\mathcal{C}_A \subseteq S \times S$ in the following way: $$(x, y)\in\mathcal{C}_A \iff \exists a \in A^\lor \text{ such that } x \lor a = y \lor a.$$ The relation $\mathcal{C}_A$ is a \emph{join-semilattice congruence} on $S$ (i.e. a congruence on the join-semilattice reduct $(S, \lor)$). The assignment 
\begin{align*} 
    A \mapsto A^\mathrm{dj} &:= \{x \in S \mid \exists a \in A^\lor \text{ such that } (x, a)\in\mathcal{C}_A\} \\ &\phantom{:}= \{x \in S \mid \exists a \in A^\lor \text{ such that } x \lor a = a\}
\end{align*}
defines a closure operator on $S$ (called \emph{downward-join closure}). 
\begin{prop}\label{equivalent_form_of_downward_join_closure}
Let $S$ be a semilattice-ordered semigroup. Then for every subset $A \subseteq S$,  
\begin{equation*} 
    A^\mathrm{dj} = \downarrow(A^\lor) = \{x\in S\mid \exists a_1, \ldots, a_n \in A \text{ such that } x \leqslant a_1 \lor \ldots \lor a_n\}. 
\end{equation*}
\end{prop}
\begin{proof}
    Let $x \in S$ and $A \subseteq S$. Then 
    \begin{align*}
        x \in A^\mathrm{dj} &\iff x \in \{y \in S \mid \exists a \in A^\lor \text{ such that } y \lor a = a\} \\ &\iff \exists a\in A^\lor \text{ such that } x \lor a = a \\ & \iff \exists a \in A^\lor \text{ such that } x \leqslant a \\ &\iff x \in \downarrow(A^\lor). \qedhere 
    \end{align*}
\end{proof}
\noindent It follows that the downward-join closure is described only in terms of the order and the join structure. In fact, $A^\mathrm{dj}$ is the least subset of $S$ containing $A$ that is both a downset and closed under finite joins. 
\begin{prop}\label{downward_join_closure_is_algebraic}
    Let $S$ be a semilattice-ordered semigroup. The closure operator $(-)^\mathrm{dj} : \mathcal{P}(S) \rightarrow\mathcal{P}(S)$ is algebraic and multiplicative.
\end{prop}
\begin{proof}
    Algebraicity is immediate from Proposition \ref{equivalent_form_of_downward_join_closure}. Indeed, if $x \in A^\mathrm{dj}$, then $$x \leqslant a_1 \lor \ldots \lor a_n$$ for some $a_1, \ldots, a_n \in A$, and therefore $x \in \{a_1, \ldots, a_n\}^\mathrm{dj}.$ Hence $$A^\mathrm{dj} = \bigcup\{F^\mathrm{dj} \mid F \in \mathcal{P}_\mathrm{fin}(A)\}.$$
    
    We now prove multiplicativity. Let $A, B \subseteq S$. It is enough to show that $$(A^\mathrm{dj}B^\mathrm{dj})^\mathrm{dj} \subseteq (AB)^\mathrm{dj}.$$ Let $x \in (A^\mathrm{dj}B^\mathrm{dj})^\mathrm{dj}$. Then $$x \leqslant a_1'b_1' \lor \ldots \lor a_n'b_n'$$ for some $a_1', \ldots, a_n' \in A^\mathrm{dj}$ and $b_1', \ldots, b_n' \in B^\mathrm{dj}$. For every $i \in \{1, \ldots, n\}$ there exist $a_{i1}, \ldots, a_{im_i} \in A$, $b_{i1}, \ldots, b_{ik_i} \in B$ such that $$a_i' \leqslant a_{i1} \lor \ldots \lor a_{im_i}, \qquad b_i' \leqslant b_{i1} \lor \ldots \lor b_{ik_i}$$ and hence $$a_i'b_i' \leqslant (a_{i1} \lor \ldots \lor a_{im_i})(b_{i1} \lor \ldots \lor b_{ik_i}) = \bigvee_{\substack{1 \leqslant p \leqslant m_i \\1 \leqslant q \leqslant k_i}}a_{ip}b_{iq}.$$ Each product $a_{ip}b_{iq} \in AB$, so the right-hand side belongs to $(AB)^\lor$. Therefore, $a_i'b_i' \in (AB)^\mathrm{dj}$ for every $i\in\{1, \ldots, n\}$ and in consequence $$a_1'b_1' \lor \ldots \lor a_n'b_n' \in (AB)^\mathrm{dj}.$$ Since $(AB)^\mathrm{dj}$ is a downset, we conclude that $x \in (AB)^\mathrm{dj}$. Thus, $(A^\mathrm{dj}B^\mathrm{dj})^\mathrm{dj} \subseteq (AB)^\mathrm{dj}$, as required.  
\end{proof}

\subsection{Join-radical closure}
The next closure is inspired by the use of finite joins of successive powers in the study of semilattice-ordered semigroups. This point of view already appears in the work of Calude on topological properties of semilattice-ordered semigroups \cite{calude1976_1}. 

Following \cite[Section 1]{calude1976_1}, let $S$ be a semilattice-ordered semigroup and for $x\in S$ we set $x^1 := x$, $x^{n + 1} := x^n \cdot x$ and $x^{(n)} := x \lor x^2 \lor \ldots \lor x^n$. Calude defined a closure operator by the assignment $$A \mapsto \{x^{(n)} \mid x \in A,\hspace{3pt} n\in\mathbb{N}_+\}.$$ In the present paper we consider a different closure operator based on the same family of elements $x^{(n)}$. Before introducing this closure, we establish elementary properties of elements $x^{(n)}$. 

\begin{prop}\label{power_proposition}
    Let $S$ be a semilattice-ordered semigroup and let $x,y \in S$ be such that $x \leqslant y$. Then for every $n \in \mathbb{N}_+$, $x^{(n)} \leqslant y^{(n)}$. 
\end{prop}
\begin{proof}
    We first show by induction on $n \geqslant 1$ that $x^n \leqslant y^n$. For $n = 1$, we have $x \leqslant y$ and it is clearly true. Now suppose that $x^k \leqslant y^k$ for some $k \geqslant 1$. Then 
    \begin{equation*}
        x^{k + 1} = x^kx \leqslant y^ky = y^{k + 1},
    \end{equation*}
    which establishes the induction step. Hence $x^n \leqslant y^n$ for all $n \geqslant 1$.
    
    Now from $x^i \leqslant y^i$ for each $i = 1,2,\ldots, n$ we get 
    \begin{equation*}
        x^{(n)} = x^1 \lor x^2 \lor \ldots \lor x^n \leqslant y^1 \lor y^2 \lor \ldots \lor y^n = y^{(n)}. \qedhere
    \end{equation*}
\end{proof}

\begin{prop}\label{product_of_powers}
    Let $S$ be a semilattice-ordered semigroup and $x\in S$. Then for every $n, m \in \mathbb{N}_+$, $(x^{(n)})^{(m)} = x^{(nm)}$. 
\end{prop}
\begin{proof}
    In a semilattice-ordered semigroup multiplication distributes over the join operation $\lor$ and therefore for every $k \geqslant 1$,
    \begin{equation*}
        (x^{(n)})^k = (x \lor x^2 \lor \ldots \lor x^n)^k = \displaystyle\bigvee_{1 \leqslant i_1, \ldots, i_k \leqslant n} x^{i_1}\ldots x^{i_k} = \displaystyle\bigvee_{i = k}^{kn} x^i. 
    \end{equation*}
    We then have 
    \begin{equation*}
        (x^{(n)})^{(m)} = \displaystyle\bigvee_{k = 1}^m (x^{(n)})^k = \bigvee_{k = 1}^m\bigvee_{i = k}^{kn} x^i = \bigvee_{i = 1}^{nm} x^i = x^{(nm)}. \qedhere 
    \end{equation*}
\end{proof}
\noindent The assignment 
\begin{equation*}
    A \mapsto \sqrt[\lor]{A} := \{x \in S \mid \exists n\in\mathbb{N}_+ \text{ such that } x^{(n)} \in A\}
\end{equation*}
defines a closure operator on $S$ (called \emph{join-radical closure}). Indeed, extension and order-preservation of this assignment are immediate, whereas idempotence follows from Proposition \ref{product_of_powers}. 
\begin{lemma}\label{join_radical_doesnt_change_downsets}
    Let $S$ be a semilattice-ordered semigroup. If $A\subseteq S$ is a downset, then $\sqrt[\lor]{A} = A$. 
\end{lemma}
\begin{proof}
    Clearly $A \subseteq \sqrt[\lor]{A}$. If $x \in \sqrt[\lor]{A}$, then $x^{(n)} \in A$ for some $n \in \mathbb{N}_+$. Since $A$ is a downset and $x \leqslant x^{(n)}$, it follows that $x \in A$. Thus, $\sqrt[\lor]{A} \subseteq A$.  
\end{proof}

\begin{prop}\label{join_radical_closure_is_algebraic}
    Let $S$ be a semilattice-ordered semigroup. The closure operator $\sqrt[\lor]{-} : \mathcal{P}(S) \rightarrow \mathcal{P}(S)$ is algebraic. 
\end{prop}
\begin{proof}
    Let $A$ be a subset of $S$. If $x \in \sqrt[\lor]{A}$, then $x^{(n)} \in A$ for some $n \in \mathbb{N}_+$. Since $x \in \sqrt[\lor]{\{x^{(n)}\}}$ and $\{x^{(n)}\} \in \mathcal{P}_\mathrm{fin}(A)$, we have 
    \begin{equation*}
        \sqrt[\lor]{A} \subseteq \bigcup \big\{\sqrt[\lor]{F} \mid F \in \mathcal{P}_\mathrm{fin}(A)\big\}.  
    \end{equation*}
    The reverse inclusion follows from order-preservation of $\sqrt[\lor]{-}$. 
\end{proof}
\begin{example}\label{join_radical_closure_is_not_multiplicative}
    Let $S := \mathbb{N}_+$ with multiplication given by the ordinary product of natural numbers and join operation defined as follows: $$x \lor y := \max(x,y).$$ Then $(S, \cdot, \lor)$ is a semilattice-ordered semigroup. For every $x, n\in\mathbb{N}_+$, $$x^{(n)} = x \lor x^2 \lor \ldots \lor x^n = x^n$$ and hence $$\sqrt[\lor]{A} = \{x \in \mathbb{N}_+ \mid \exists n \in\mathbb{N}_+ \text{ such that } x^n \in A\}$$ for each subset $A \subseteq \mathbb{N}_+$. Now let $A := \{4\}$ and $B := \{9\}$. Then $$\sqrt[\lor]{AB} = \{6, 36\}, \qquad \sqrt[\lor]{\sqrt[\lor]{A}\sqrt[\lor]{B}} = \{6, 12, 18, 36\},$$
    so the join-radical closure is not multiplicative. 
\end{example}

\subsection{Join-bi-ideal closure} 
Bi-ideals are classical objects in semigroup theory. They were introduced by Good and Hughes in \cite{good1952associated} in the setting of semigroups and extended by Ponizovskii, Kehayopulu and Tsingelis to ordered semigroups in \cite{Kehayopulu}. Since then, bi-ideals have been studied extensively in numerous contexts, including fuzzy bi-ideals \cite{KEHAYOPULU200513, Gaketem2023ONEB}, intuitionistic fuzzy bi-ideals \cite{Jun, Yongfa2007}, fuzzy generalized bi-ideals \cite{Khan} and the algebraic structure of the family of all bi-ideals \cite{Mallick}. 

In the setting of ordered semigroups, a bi-ideal is usually defined as a subsemigroup $B$ of a semigroup $S$ that is downward closed and satisfies the condition $BSB \subseteq B$. On the other hand, in a semilattice-ordered semigroup it is natural to require ideal-like subsets to be closed under finite joins. This motivates the following definition.  
\begin{definition}\label{definition_join_bi_ideal}
    Let $S$ be a semilattice-ordered semigroup. A \emph{join-bi-ideal} of $S$ is a subset $B \subseteq S$ such that: 
    \begin{enumerate}[label=(\arabic*)]
        \item $\downarrow B = B$, 
        \item $B^{\lor} = B$, 
        \item $BB \subseteq B$, 
        \item $BSB \subseteq B$. 
    \end{enumerate}
\end{definition}

It can be easily verified that each condition in Definition \ref{definition_join_bi_ideal} is preserved under arbitrary intersections, so we obtain the following: 
\begin{prop}
    The intersection of any family of join-bi-ideals is a join-bi-ideal. 
\end{prop}

\noindent Consequently, for each subset $A \subseteq S$, we may define
\begin{equation*}
    \mathcal{B}(A) := \bigcap\{B \subseteq S \mid B \text { is a join-bi-ideal and } A \subseteq B\}. 
\end{equation*}
Equivalently, $\mathcal{B}(A)$ is the smallest join-bi-ideal containing $A$. The assignment $$A \mapsto \mathcal{B}(A)$$ defines a closure operator on $S$ (called \emph{join-bi-ideal closure}). 
\begin{prop}\label{join_radical_is_contained_in_bi_ideal}
    Let $S$ be a semilattice-ordered semigroup. Then for every subset $A \subseteq S$, $$\sqrt[\lor]{A} \subseteq A^\mathrm{dj} \subseteq \mathcal{B}(A).$$ 
\end{prop}
\begin{proof}
    Clearly $A \subseteq A^\mathrm{dj}$ and hence $\sqrt[\lor]{A} \subseteq \sqrt[\lor]{A^\mathrm{dj}}$. Since $A^\mathrm{dj}$ is in particular a downset, Lemma \ref{join_radical_doesnt_change_downsets} yields $\sqrt[\lor]{A} \subseteq \sqrt[\lor]{A^\mathrm{dj}} = A^\mathrm{dj}$. Since $\mathcal{B}(A)$ is a join-bi-ideal containing $A$, it is in particular a downset closed under finite joins. Therefore, $A^\mathrm{dj} = \downarrow(A^\lor) \subseteq \mathcal{B}(A).$
\end{proof}

\begin{prop}
    Let $S$ be a semilattice-ordered semigroup. Then for every $A \subseteq S$, 
    \begin{enumerate}[label=(\arabic*)]
        \item $\sqrt[\lor]{\mathcal{B}(A)} = \mathcal{B}(A) = \mathcal{B}({\sqrt[\lor]{A}})$,
        \item $\sqrt[\lor]{A^\mathrm{dj}} = A^\mathrm{dj} = (\sqrt[\lor]{A})^\mathrm{dj}$.
    \end{enumerate} 
\end{prop}
\begin{proof}
    (1) Let $A$ be a subset of $S$. By Proposition \ref{join_radical_is_contained_in_bi_ideal}, we have $\sqrt[\lor]{A} \subseteq \mathcal{B}(A)$. Applying closure $\mathcal{B}$, we get $$\mathcal{B}(\sqrt[\lor]{A}) \subseteq \mathcal{B}(\mathcal{B}(A)) = \mathcal{B}(A).$$ On the other hand, $A \subseteq \sqrt[\lor]{A}$ and hence $\mathcal{B}(A) \subseteq \mathcal{B}(\sqrt[\lor]{A})$. Therefore, $\mathcal{B}(A) = \mathcal{B}(\sqrt[\lor]{A})$. Since $\mathcal{B}(A)$ is a join-bi-ideal, it is in particular a downset. By Lemma \ref{join_radical_doesnt_change_downsets}, we obtain $$\sqrt[\lor]{\mathcal{B}(A)} = \mathcal{B}(A) = \mathcal{B}(\sqrt[\lor]{A}),$$ as required. \\
    (2) In this case the proof is analogous to the proof of (1). 
\end{proof}

\begin{prop}\label{join_bi_ideal_closure_is_algebraic}
    Let $S$ be a semilattice-ordered semigroup. The closure operator $\mathcal{B} : \mathcal{P}(S) \rightarrow \mathcal{P}(S)$ is algebraic. 
\end{prop}
\begin{proof}
    Let $A$ be a subset of $S$ and $B := \bigcup\{\mathcal{B}(F) \mid F \in \mathcal{P}_\mathrm{fin}(A)\}$. Clearly $A \subseteq B$, since for every $a \in A$ we have $a \in \mathcal{B}(\{a\}) \subseteq B$. We claim that $B$ is a join-bi-ideal. 
    
    If $x \in B$ and $y \leqslant x$, then $x \in \mathcal{B}(F)$ for some finite subset $F \subseteq A$. Since $\mathcal{B}(F)$ is a downset, $y\in\mathcal{B}(F) \subseteq B$ and thus $B$ is a downset. Now take elements $z, w \in B$. There exist finite subsets $G, H \subseteq A$ such that $z \in \mathcal{B}(G)$ and $w \in \mathcal{B}(H)$. Define a finite set $K := G \cup H$. By order-preservation of $\mathcal{B}$, we have $\mathcal{B}(G), \mathcal{B}(H) \subseteq \mathcal{B}(K)$ and hence $z, w \in \mathcal{B}(K)$. Then we get $z \lor w \in \mathcal{B}(K)$, $zw \in \mathcal{B}(K)$ and $zsw \in \mathcal{B}(K)$ for any element $s \in S$. Since $\mathcal{B}(K) \subseteq B$, we conclude that $B$ is a join-bi-ideal. 
    
    By minimality of $\mathcal{B}(A)$, we get $\mathcal{B}(A) \subseteq B$. The reverse inclusion $B \subseteq \mathcal{B}(A)$ holds because for every finite subset $F \subseteq A$ we have $\mathcal{B}(F) \subseteq \mathcal{B}(A)$ by order-preservation of $\mathcal{B}$. Hence $\mathcal{B}(A) = B$, proving algebraicity. 
\end{proof}
\begin{example}
    Let $(L, \lor, \land)$ be a distributive lattice and put $S := L$ with $x \cdot y := x \land y$. Then $(S, \cdot, \lor)$ is a semilattice-ordered semigroup. In this setting a subset $B \subseteq L$ is a join-bi-ideal if and only if it is a possibly empty lattice ideal. Hence for every $A \subseteq L$ we have
    \begin{equation*}
        \mathcal{B}(A) = \downarrow(A^\lor) = \{x \in L \mid \exists a_1, \ldots, a_n \in A \text{ such that } x \leqslant a_1 \lor \ldots \lor a_n\}.
    \end{equation*}
    Consequently, 
    \begin{equation*}
        X = \{I \subseteq L \mid I = \downarrow I = I^\lor\}  
    \end{equation*}
    is precisely the set of all ideals of $L$, endowed with the induced hull-kernel topology. Moreover, for $I, J \in X$ we have 
    \begin{equation*}
        I \star J = \mathcal{B}(IJ) = \mathcal{B}(\{x \land y \mid x \in I,\hspace{3pt} y \in J\}) = I \cap J. 
    \end{equation*}
    Indeed, $I \cap J$ is an ideal and contains all meets $x \land y$. Conversely, every element of $I \cap J$ is a meet of itself with itself. 
    
    We claim that $\star$ is spectral. Indeed, for every finite $F \subseteq L$, 
    \begin{equation*}
        F \subseteq I \star J \iff F \subseteq I \cap J \iff F \subseteq I \text{ and } F \subseteq J. 
    \end{equation*}
    Thus, the condition of Proposition \ref{criterion_for_joint_spectrality} holds with a single pair $(G, H) = (F, F)$. Hence $\star$ is spectral. 
\end{example}
\begin{example}\label{join_bi_ideal_closure_is_not_multiplicative}
    Let $S := \{0,1,2\}$ with multiplication given by 
    \begin{table}[H]
        \centering
            \begin{tabular}{c|ccc}
                $\cdot$ & 0 & 1 & 2 \\ \hline
                0 & 0 & 0 & 0 \\
                1 & 0 & 1 & 2 \\
                2 & 2 & 2 & 2 
            \end{tabular}
    \end{table}
    and join operation defined as follows: $$x\lor y := \max(x,y).$$ Then $(S, \cdot, \lor)$ is a semilattice-ordered semigroup. The sets $\varnothing$, $\{0\}$ and $S$ are the only join-bi-ideals of $S$. Let $A := \{1\}$ and $B := \{0\}$. Then $$\mathcal{B}(AB) = \{0\}, \qquad \mathcal{B}\big(\mathcal{B}(A)\mathcal{B}(B)\big) = S,$$ so the join-bi-ideal closure is not multiplicative. 
\end{example}
The three closure operators considered in this section have different multiplicative behaviour. All of them are algebraic, whereas the downward-join closure is always multiplicative and the join-radical and join-bi-ideal closures need not be multiplicative. Nevertheless, algebraicity alone is enough for the associated spaces of closed sets to be spectral and for general spectrality criteria of Section \ref{topological_semigroups_whose_underlying_space_is_spectral} to apply. We now use these observations in several concrete situations. 
\section{Applications}\label{applications}
In this section we apply the general results of Sections \ref{spectral_spaces_arising_from_closures} and \ref{topological_semigroups_whose_underlying_space_is_spectral} to the closure operators introduced in the preceding section. For these closure operators, we write $$X_\mathrm{dj} := \{A \in \mathcal{P}(S) \mid A^\mathrm{dj} = A\}, \qquad X_{\sqrt[\lor]{\ }} := \{A \in \mathcal{P}(S) \mid \sqrt[\lor]{A} = A\}$$ and $$X_{\mathcal{B}} := \{A \in\mathcal{P}(S) \mid \mathcal{B}(A) = A\}.$$ We begin by comparing the three spaces defined above.
\begin{prop}
    Let $S$ be a semilattice-ordered semigroup. Then $$X_\mathcal{B} \subseteq X_\mathrm{dj} \subseteq X_{\sqrt[\lor]{\ }}$$ and both inclusion maps $$X_\mathcal{B} \hookrightarrow X_\mathrm{dj}, \qquad X_\mathrm{dj} \hookrightarrow X_{\sqrt[\lor]{\ }}$$ are spectral.
\end{prop}
\begin{proof}
    The inclusions $X_\mathcal{B} \subseteq X_\mathrm{dj} \subseteq X_{\sqrt[\lor]{\ }}$ follow clearly from Proposition \ref{join_radical_is_contained_in_bi_ideal}, so it remains to show that the inclusion maps are spectral. We prove it for $\iota : X_\mathcal{B} \hookrightarrow X_\mathrm{dj}$, since the proof for $X_\mathrm{dj} \hookrightarrow X_{\sqrt[\lor]{\ }}$ is identical. Let $F$ be a finite subset of $S$. Then $$\iota^{-1}\big(D_{X_\mathrm{dj}}(F)\big) = D_{X_\mathrm{dj}}(F) \cap X_\mathcal{B} = D_{X_\mathcal{B}}(F).$$ Since $D_{X_\mathcal{B}}(F)$ is quasi-compact and open, and every quasi-compact open subset of $X_\mathrm{dj}$ is a finite union of finite intersections of sets of the form $D_{X_\mathrm{dj}}(F)$, the map $\iota$ is spectral. 
\end{proof}

\begin{prop}
    Let $S$ be a semilattice-ordered semigroup. If the poset $(S, \leqslant)$ is well-quasi-ordered, then the maps $$\star_\mathrm{dj} : X_\mathrm{dj} \times X_\mathrm{dj} \rightarrow X_\mathrm{dj}, \qquad A \star_\mathrm{dj} B = (AB)^\mathrm{dj},$$ and $$\star_\mathcal{B} : X_\mathcal{B} \times X_\mathcal{B} \rightarrow X_\mathcal{B}, \qquad A \star_\mathcal{B} B = \mathcal{B}(AB),$$ are spectral. 
\end{prop}
\begin{proof}
    Every $\mathcal{B}$-closed subset of $S$ and every $\mathrm{dj}$-closed subset of $S$ are downsets. By Proposition \ref{order_compatibility_is_equivalent_to_all_closed_subsets_being_down_sets}, both these closure operators are order-compatible. The conclusion now follows from Corollary \ref{algebraicity_order_compatibility_and_well_quasi_order_on_original_set_imply_spectrality_of_star}. 
\end{proof}
The preceding proposition gives a useful sufficient condition for spectrality of the map $\star: X \times X \to X$. In general, however, spectrality of all left and right translations does not imply spectrality of the binary operation itself. The following example shows that this may occur even for an algebraic multiplicative closure operator, namely for the downward-join closure introduced above. 
\begin{example}\label{separate_spectrality_does_not_imply_joint}
    Let $$E := \{d, \infty\} \cup \{e_n \mid n \in \mathbb{N}\}$$ be the join-semilattice in which $d < e_n < \infty$ for every $n\in\mathbb{N}$ and the elements $e_n$ are pairwise incomparable. We next define a join-semilattice $Z$. Put $$Z := \{0, 1\} \cup \{r_n, s_n \mid n \in \mathbb{N}\} \cup \{t_{nm} \mid n, m\in\mathbb{N}, n \neq m\}$$ with the order determined by $$0 < r_n < t_{nm} < 1, \qquad 0 < s_m < t_{nm} < 1$$ for all distinct $n, m\in\mathbb{N}$. Define join-homomorphisms $\rho, \sigma: E \to Z$ by
    $$\rho(x) = \begin{cases}
        0, & \text{for } x = d, \\
        r_n, & \text{for } x = e_n, \\
        1, & \text{for } x = \infty, 
    \end{cases}
    \qquad \sigma(x) = \begin{cases}
        0, & \text{for } x = d, \\ 
        s_n, & \text{for } x = e_n, \\ 
        1, & \text{for } x = \infty. 
    \end{cases}$$ Define a map $\mu: E \times E \to Z$ by $$\mu(x, y) := \rho(y) \lor \sigma(x).$$ Since $\rho$ and $\sigma$ preserve joins, the map $\mu$ preserves joins in each variable. Let $S$ be the disjoint union $E \sqcup Z$, ordered by the given orders on $E$ and $Z$ and by declaring $z < e$ for all $z\in Z$, $e\in E$. Define multiplication on $S$ by 
    \begin{equation}\label{definition_of_multiplication}
        x\cdot y = \begin{cases}
        \mu(x, y), & \text{for } x, y \in E, \\ 
        0, & \text{otherwise.}
    \end{cases}
    \end{equation}
    The multiplication $\cdot$ is associative because every product of three elements of $S$ is equal to $0$. We next verify distributivity over joins. Let $x, y, z\in S$. If $x, y, z\in E$, then $$x(y \lor z) = \mu(x, y\lor z) = \mu(x, y) \lor \mu(x, z) = xy \lor xz,$$ since $\mu$ preserves joins in its second variable. If $x \in Z$ or both $y, z\in Z$, then $x(y \lor z) = xy = xz = 0$. Finally, if $x, y \in E$ and $z \in Z$, then $y \lor z = y$, because every element of $Z$ lies below every element of $E$. Hence $$x(y \lor z) = xy = xy \lor 0 = xy \lor xz.$$ The remaining case, where $x, z\in E$ and $y\in Z$, is analogous. The equality $(x \lor y)z = xz \lor yz$ also holds, as $\mu$ preserves joins in its first variable as well. Therefore, $(S, \cdot, \lor)$ is a semilattice-ordered semigroup. 
    
    For every subset $\varnothing \neq A \in \mathcal{P}(E)$, $A^{\mathrm{dj}}$ is of the form $\downarrow e$, where $e \in E$. Similarly, for every subset $\varnothing \neq B \in \mathcal{P}(Z)$, $B^\mathrm{dj}$ is of the form $\downarrow z$, where $z\in Z$. Now consider the downward-join closure on $S$. It can be easily seen that its corresponding spectral space $X$ is of the form $$X = \{\varnothing\} \cup \{\downarrow a \mid a \in S\}.$$ Since multiplication $\cdot: S \times S \to S$ is order-preserving, we get \begin{equation}\label{property_of_multiplication_on_X}
        (\downarrow a) \star (\downarrow b) = \downarrow(ab) 
    \end{equation}
    for all $a, b \in S$. 

    We now prove that every right translation is spectral. Since for every nonempty finite subset $F \subseteq S$ we have $$D_X(F) = D_X\Big(\bigvee F\Big),$$ it is enough to determine inverse images of sets $D_X(s)$, $s\in S$. Fix an element $J \in X$. If $J = \varnothing$, then $R_J$ is constantly equal to $\varnothing$, so it is spectral. Suppose that $J = \downarrow y$. If $y \in Z$, then by \eqref{definition_of_multiplication} and \eqref{property_of_multiplication_on_X}, $$R_J^{-1}\big(D_X(0)\big) = D_X(0) = \{\varnothing\}$$ and $$R_J^{-1}\big(D_X(s)\big) = X$$ for every $0 \neq s\in S$. Thus, $R_J$ is spectral. It remains to consider the case $y\in E$. For $z \in Z \setminus\{0\}$ put $$M_z(y) := \{x \in E \mid z \leqslant \mu(x, y)\}.$$ We claim that $M_z(y)$ is a finitely generated upset of $E$, with at most two minimal elements. If $y = d$, then $$M_z(d) = \begin{cases}
        \uparrow e_m, & \text{for } z = s_m, \\
        \uparrow\infty, & \text{otherwise}. 
    \end{cases}$$ If $y = \infty$, then $M_z(\infty) = E = \uparrow d$ for every $z \in Z\setminus\{0\}$. Finally, let $y = e_j$. If~$z \leqslant r_j$, then $M_z(e_j) = E = \uparrow d$. Assume that $z \not\leqslant r_j$. There exists at most one index $i \neq j$ such that $z \leqslant t_{ji}$. Therefore, it can be easily seen that $M_z(e_j) = \uparrow e_j$ or $M_z(e_j) = \uparrow\{e_j, e_i\}$, which proves the claim. Thus, we may choose a finite set $G_z(y)\subseteq E$ such that $|G_z(y)| \leqslant 2$ and $$M_z(y) = \bigcup_{g \in G_z(y)}\uparrow g.$$

    Let $s \in S$. If $s\in E$, then $s$ does not belong to any product $I \star J$ and in consequence $R_J^{-1}(D_X(s)) = X$. We also have $R_J^{-1}(D_X(0)) = D_X(0)$. Finally, take $z \in Z\setminus\{0\}$. We claim that for every $I \in X$, 
    \begin{equation}\label{claimed_equivalence}
    z \in I \star J \iff g \in I \text{ for some } g\in G_z(y).
    \end{equation}
    If $I = \varnothing$, the condition \eqref{claimed_equivalence} is immediate. If $I = \downarrow a$, where $a \in Z$, then $I \star J = \downarrow 0$, while $I$ does not contain any element of $E$. Thus, the condition \eqref{claimed_equivalence} holds again. Finally, if $I = \downarrow x$, where $x \in E$, then 
    \begin{align*}
        z \in I \star J &\iff z \leqslant \mu(x, y) \\ & \iff x\in M_z(y) \\ & \iff g \leqslant x \text{ for some } g\in G_z(y) \\ & \iff g\in I\text{ for some } g \in G_z(y). 
    \end{align*}
    Therefore, we get $$R_J^{-1}\big(D_X(z)\big) = \bigcap_{g\in G_z(y)} D_X(g),$$ which is quasi-compact and open. Thus, $R_J$ is spectral. 
    
    The proof for left translations is analogous, so it remains to show that $\star$ is not spectral. Assume towards a contradiction that $\star$ is spectral. Put $I_n := \downarrow e_n$ for each $n \in \mathbb{N}$ and $I_d := \downarrow d$. Since $I_d \star I_d = \downarrow 0 = \{0\}$, we have $1 \not\in I_d \star I_d$. By Proposition \ref{criterion_for_joint_spectrality}, there exists a finite family of pairs $(G_i, H_i)\in\mathcal{P}_\mathrm{fin}(S)\times\mathcal{P}_\mathrm{fin}(S)$, $i\in I$, such that for all $A, B \in X$, 
    \begin{equation}\label{finite_criterion_for_joint_spectrality}
        1 \in A \star B \iff \exists i\in I\text{ such that } G_i\subseteq A\text{ and } H_i\subseteq B.
    \end{equation}
    For every $n\in\mathbb{N}$, since $1 \in I_n\star I_n$, there exists $i_n \in I$ such that $G_{i_n}\subseteq I_n$ and $H_{i_n}\subseteq I_n$. By the pigeonhole principle, there exist an index $j$ and an infinite subset $W \subseteq \mathbb{N}$ such that $G_j\subseteq I_n$ and $H_j\subseteq I_n$ for every $n \in W$. Since $\bigcap_{n \in W}I_n = I_d$, we have $G_j\subseteq I_d$ and $H_j \subseteq I_d$. Applying \eqref{finite_criterion_for_joint_spectrality}, we obtain $1 \in I_d\star I_d$, a contradiction. Therefore, the map $\star: X \times X \to X$ is not spectral. 
\end{example}
\bibliographystyle{plain}
\bibliography{references}
\end{document}